\documentclass[reqno,12pt]{article} 

\usepackage{amsmath}
\usepackage{amssymb}
\usepackage{amsbsy}
\usepackage{fancybox}
\usepackage{setspace}
\usepackage[pdftex]{graphicx} 
\usepackage{subfig}

\usepackage[top=1in, bottom=1in, left=1in, right=1in]{geometry} 

\newtheorem{theorem}{Theorem}[section]

\newtheorem{cor}{Corollary}[section]

\newtheorem{lemma}{Lemma}[section]

\newcommand{\qed}{\hfill\rule{2.1mm}{2.1mm}}

\newcommand{\R}{\mathbf{R}}

\renewcommand{\P}{\mathbf{P}}
\newcommand{\eref}[1]{$(\ref{#1})$}

\renewcommand{\b}{\mathbf{b}}
\newcommand{\blambda}{\mathbf{\lambda}}
\newcommand{\boldeta}{\mathbf{\eta}}

\renewcommand{\d}{\mathbf{d}}
\newcommand{\e}{\mathbf{e}}
\newcommand{\f}{\mathbf{f}}
\newcommand{\n}{\mathbf{n}}
\renewcommand{\u}{\mathbf{u}}
\newcommand{\bv}{\mathbf{v}}
\renewcommand{\v}{\mathbf{v}}
\newcommand{\x}{\mathbf{x}}
\newcommand{\y}{\mathbf{y}}
\newcommand{\z}{\mathbf{z}}

\newcommand{\diag}{\mathrm{diag}\,}
\newcommand{\Diag}{\mathrm{Diag}\,}

\newcommand{\rank}{\mathrm{rank}\,}
\newcommand{\st}{\mathrm{s.t.}\;}
\newcommand{\tr}{\mathrm{tr}\,}

\newcommand{\ra}{\rightarrow}

\title{Convex optimization for the planted k-disjoint-clique problem
\thanks{Supported in part by a Discovery Grant and Postgraduate Scholarship (Doctoral) from
NSERC (Natural Science and Engineering Research Council of Canada), MITACS (Mathematics of Information Technology and Complex Systems), and the US Air Force Office of Scientific Research.}}
\author{Brendan P.W.~Ames
\thanks{Institute for Mathematics and its Applications,
  College of Science and Engineering,
  University of Minnesota,
  207 Church Street SE,
  400 Lind Hall,
  Minneapolis, Minnesota, 55455, U.S.A., bpames@gmail.com} \and
Stephen A.~Vavasis\thanks{Department of
Combinatorics and Optimization, University of Waterloo, 200
University Avenue W., Waterloo, Ontario N2L 3G1, Canada,
vavasis@math.uwaterloo.ca}}

\begin{document}
\maketitle


We consider the {\em $k$-disjoint-clique} problem.  The input
is an undirected graph $G$ in which the nodes represent data items, and edges
indicate a similarity between the corresponding items.  The problem is
to find within the graph
$k$ disjoint cliques
that cover the maximum number of nodes of $G$.
This problem may be understood as a general way
to pose the classical `clustering' problem.  In clustering, one is given
data items and a distance function, and one wishes to partition the data into
disjoint clusters of data items, such that the items in each cluster are close
to each other.  Our formulation additionally allows `noise' nodes to be present
in the input data that are not part of any of the cliques.

The $k$-disjoint-clique problem is NP-hard, but we show that a convex relaxation can
solve it in polynomial time for input instances constructed in a certain way.  The
input instances for which our algorithm finds the optimal
solution consist of $k$ disjoint large cliques
(called `planted cliques') that are then obscured by noise edges inserted
either at random or by an adversary, as well as additional nodes not belonging to any of the $k$ planted cliques.

\section{Introduction}
\label{sec:intro}

Given a set of data, \textit{clustering} seeks to partition the data into sets of similar objects.
These subsets are called `clusters', and the goal is to find a few large clusters covering as much of the data as possible. Clustering plays a significant role in a wide range of applications;
including, but not limited to, information retrieval, pattern recognition, computational biology, and image processing.
For a recent survey of clustering techniques and algorithms with a particular focus on applications in data mining see \cite{Berkhin}.

In this paper, we consider the following graph-based representation of data.
Given a set of data where each pair of objects is known to be similar or dissimilar, we consider the graph $G = (V,E)$ where
the objects in the given data set are the set of nodes of $G$ and any two
nodes are adjacent if and only if their corresponding objects are similar.
Hence, for this representation of the data, clustering is equivalent to partitioning $G$ into disjoint cliques.
Therefore, for any integer $k$, the problem of identifying $k$ clusters in the data containing the maximum number of objects
is equivalent to
the maximum node $k$-disjoint-clique problem of the corresponding graph $G$.
Given an undirected graph $G = (V,E)$ and integer $k \in [1, |V|]$, the \textit{maximum node $k$-disjoint-clique problem} refers to the problem of finding the subgraph $K$ of $G$ composed of a collection of $k$ disjoint cliques,
called a \textit{$k$-disjoint-clique subgraph}, maximizing the number of nodes in $K$.
Unfortunately, since the $k=1$ case is exactly the maximum clique problem, well-known to be NP-hard \cite{GJ},
the maximum node $k$-disjoint-clique problem is NP-hard.

In Section~\ref{sec:k-clique}, we relax the maximum node $k$-disjoint-clique problem to a semidefinite program.
We show that this convex relaxation can recover the exact solution in two cases.
In the first case, presented in Section~\ref{sec: adversarial case}, the input graph is constructed deterministically
as follows.
The input graph consists of $k$ disjoint cliques $C_1,\dots, C_k$, each of size at least $\hat r$,
plus a number of diversionary nodes and edges inserted by an adversary.
We show that the algorithm can tolerate up to $O(\hat r^2)$ diversionary edges and nodes provided that, for each $i=1,\dots, k$,
each node in the clique $C_i$ is adjacent to at most $O(\min\{|C_i|, |C_j|\})$ nodes in the clique $C_j$
for each $j=1,\dots, k$ such that $i \neq j$.
In Section~\ref{sec: random case}, we suppose that the graph contains a $k$-disjoint-clique subgraph $K$ and some additional nodes, and the remaining nonclique edges are added to the graph
independently at random with fixed probability $p$.
We give a general formula for clique sizes that can be recovered by the algorithm; for example,
if the graph contains $N$ nodes total and $N^{1/4}$ planted cliques each of size $\Omega(N^{1/2})$, then the
convex relaxation will find them.
We develop the necessary optimality and uniqueness theorems in Section~\ref{sec:k-clique}
and provide the necessary
background on random matrices in Section~\ref{sec: random matrix theory}.

The rationale for this line of analysis is that in real-world
applications of clustering, it is often the case that the sought-after
clusters are present in the input data but are hidden by the presence
of noisy data.  Therefore, it is of interest to find cases of
clustering data in which the clusters are hidden by noise and yet can
still be found in polynomial time.

Our analysis is related in an indirect manner to work on measuring
`clusterability' of data, e.g., Ostrovsky et al.~\cite{Ostro}.  In that work,
the authors prove that a certain clustering algorithm works well if
the data has $k$ `good' clusters.  Our assumptions and analysis, however,
differ substantially from \cite{Ostro} (for example, we do not require
all the data items to be placed in clusters), so there is no direct
relationship between our result and theirs.

Our results and techniques can be seen as an extension of those in \cite{Ames-Vavasis}
from the maximum clique problem to the maximum node $k$-disjoint-clique problem.
Indeed, in the $k=1$ case, our results agree with those presented in \cite{Ames-Vavasis},
as well as those found in earlier work by Alon et al.~\cite{Alon}, and
by Feige and Krauthgamer \cite{Feige:2000}.
Recent papers by Oymak and Hassibi \cite{oymak2011finding} and Jalali et al. \cite{jalali2011clustering} 
written subsequently to the initial prepublication release of this paper
extend our results to a more general model
for clustered data.
 Specifically, in \cite{oymak2011finding} and \cite{jalali2011clustering} the authors independently propose new heuristics for partitioning an input graph $G$ into densely connected subgraphs
based on the heuristic for decomposition of a matrix into low-rank and sparse components considered in \cite{chandrasekaran2009rank} and \cite{candes2009robust}.
Under certain assumptions on the input graph, Oymak and Hassibi \cite{oymak2011finding} and Jalali et al. \cite{jalali2011clustering} show that this approach
successfully recovers the correct partition of the graph $G$ into densely connected subgraphs.
 In particular, the correct partition is recovered when the input graph is composed of several disjoint cliques obscured by noise in the form of random edge additions and {\it deletions}.
It is important to note that our approach can only tolerate noise in the form of edge additions.
The lower bound on the minimum clique size ensuring exact recovery provided by \cite{oymak2011finding} and \cite{jalali2011clustering} is equal to that provided in Section~\ref{sec: random case},
although the conditions ensuring exact recovery provided in \cite{oymak2011finding} and \cite{jalali2011clustering} do not require any explicit upper bound on the number of planted cliques to be satisfied.
On the other hand, exact recovery is guaranteed for our approach in the prescence of up to $O(\hat r^2)$ diversionary nodes, while  \cite{oymak2011finding} and \cite{jalali2011clustering} do not consider noise of this form.

More generally, our results follow in the spirit of several recent papers,
in particular Recht et al. \cite{Recht-Fazel-Parrilo:2007} and Cand\`es and Recht \cite{Candes-Recht:2008},
which consider nuclear norm minimization, a special case of semidefinite
programming, as a convex relaxation of matrix rank minimization.
\textit{Matrix rank minimization} refers to the problem of finding a minimum rank solution of a given linear system.
These papers have results of the following general form.
Suppose that it is known that the constraints of the given linear system are random in some sense and that it is known that a
solution of very low rank exists.
Then the nuclear norm relaxation recovers the (unique) solution of minimum rank.
We will argue that, in the case that the graph $G$ contains a  planted $k$-disjoint-clique subgraph $K$ and not too many diversionary edges,
a rank $k$ solution, corresponding to the adjacency matrix of $K$, of a system of linear equations defined by the input graph $G$
can be recovered exactly by solving a semidefinite program.

Like many of the papers mentioned in the previous paragraph, the proof
that the convex relaxation exactly recovers the combinatorial solution
constructs multipliers to establish that the combinatorial solution
satisfies the KKT optimality conditions of the convex problem.  Herein
we introduce a new technical method for the construction of
multipliers.  In \cite{Ames-Vavasis}, the multipliers are constructed
according to simple formulas because of the fairly simple nature of
the problem.  On the other hand, in \cite{Candes-Recht:2008}, the
multipliers are constructed by projection (i.e., solving a linear
least-squares problem), which entails a quite difficult analysis.
This paper introduces a technique of intermediate complexity: we
construct the multipliers as the solution to a system of invertible
linear equations.  We show that the equations are within a certain
norm distance of much simpler (diagonal plus rank-one) linear
equations. Finally, the result is obtained from standard bounds on the
perturbation of the solution of a linear system due to perturbation in
its coefficients.

\section{The Maximum Node $k$-disjoint-clique Problem}
\label{sec:k-clique}
Let $G = (V,E)$ be a simple graph. The \textbf{maximum node $k$-disjoint-clique} problem focuses on finding $k$ disjoint cliques in $G$ such that the total number of nodes in these cliques is maximized.
We call a subgraph of $G$ composed of $k$ disjoint cliques a ``$k$-disjoint-clique subgraph". 
This problem is clearly NP-hard since  it is equivalent to the maximum clique problem in the case that $k=1$.

The problem of maximizing the number of nodes in a $k$-disjoint-clique subgraph of $G$ can be formulated as the 
following combinatorial optimization problem
\begin{align}
	\max_{S = \{\bv_1,\dots, \bv_k\}} \; 	& \sum_{i=1}^k \bv_i^T \e	\label{e: kdc prob}	\\
	\st 		& \bv_i^T \bv_j = 0, \;\;\; \forall\; i,j=1,\dots,k, \; i\neq j	\label{e: kdc orth}	\\
			& [\bv_i \bv_i^T]_{uv} = 0,	\;\;\; \mbox{ if } uv \notin E, \; u\neq v, \;\;\; \forall\; i=1,\dots,k	\label{e: kdc nonedge} \\
			& \bv_i \in \{0,1\}^V,	\;\;\; \forall \; i =1,\dots, k. \label{e: kdc binary}
\end{align}
Here $\e$ denotes the all-ones vector in $\R^N$.
A feasible solution $S = \{\bv_1,\dots, \bv_k\}$ for \eqref{e: kdc prob} is the collection of characteristic vectors
of a set of disjoint cliques of $G$. Indeed, the constraint \eqref{e: kdc orth} ensures that
the sets of nodes indexed by $S$ are disjoint and the constraint \eqref{e: kdc nonedge} ensures 
that the sets of nodes indexed by $S$ induce complete subgraphs of $G$ and, hence, are cliques of $G$.
Note that a feasible solution $S = \{\bv_1, \dots, \bv_k\}$ need not define a partition of $V$.
That is, a feasible solution need not correspond to a $k$-disjoint-clique subgraph of $G$ that contains
every node in $V$.
Unfortunately, finding the solution to a nonlinear program with binary constraints is NP-hard in general.
The formulation \eqref{e: kdc prob} may be relaxed to the rank constrained semidefinite program
\begin{equation}    \label{e: kdc rank form}
\begin{array}{ll}
     \max &    \sum_{i=1}^N \sum_{j=1}^N X_{ij} \\
     \st   &   X \e \le \e, \\
          &    X_{ij} = 0,     \;\;\;  \forall \; (i,j) \notin E \; \st i \neq j \\
          &    \rank(X) = k,     \\
          &    X \succeq 0
\end{array}
\end{equation}
where $N = |V|$, $X$ is an $N\times N$ real symmetric matrix, and 
the notation $``X \succeq 0"$ means that $X$ is positive semidefinite.
To see that \eqref{e: kdc rank form} is a relaxation of \eqref{e: kdc prob},
suppose that $\{C_1,\dots, C_k\} \subseteq V$ defines a $k$-disjoint-clique subgraph of $G$.
Let $S = \{\bv_1, \dots, \bv_k\}$ be the set of characteristic vectors of $\{C_1,\dots, C_k\}$.
The matrix 
\begin{equation}	\label{eqn: proposed solution}
	X = \sum_{i=1}^k \frac{ \bv_i\bv_i^T}{|C_i|}
\end{equation}
is positive semidefinite with rank equal to $k$.
Note that $\sum_{i=1}^k \bv_i \le \e$ since $\bv_1,\dots, \bv_k$ are orthogonal binary vectors.
It follows that
\begin{equation}	\label{e: kdc X rows}
	X \e = \sum_{i=1}^k \left(\frac{\bv_i}{|C_i|} \right) \bv_i^T \e = \sum_{i=1}^k \bv_i \le \e
\end{equation}
since $\bv_i^T\e = |C_i|$.
Moreover,
\begin{equation}
	\sum_{i=1}^N \sum_{j=1}^N X_{ij} = \e^T X \e = \sum_{i=1}^k \bv_i^T \e.
\end{equation}
Therefore, every feasible solution $S$ for \eqref{e: kdc prob} defines
a feasible solution of \eqref{e: kdc rank form} with objective value equal to the 
number of nodes in the $k$-disjoint-clique subgraph defined by $S$.
The nonconvex  program \eqref{e: kdc rank form} may be relaxed further to a
semidefinite program by replacing the nonconvex constraint $\rank(X) =k$ with the linear 
constraint $\tr(X) = k$:

\begin{equation}    \label{eqn: k-clique relaxation}
\begin{array}{ll}
     \max &    \sum_{i=1}^N \sum_{j=1}^N X_{ij} \\
     \st   &   X \e \le \e, \\
          &    X_{ij} = 0,    \;\;\;    \forall \; (i,j) \notin E \; \st i \neq j \\
          &    \tr(X) = k,     \\
          &    X \succeq 0.
\end{array}
\end{equation}
For every $k$-disjoint-clique subgraph of $G$ composed of cliques of sizes $r_1, \dots, r_k$, \eqref{eqn: proposed solution} defines a feasible solution
$X^* \in \R^{N \times N}$ for (\ref{eqn: k-clique relaxation}) such that
$$
     X^*_{ij} = \left\{ \begin{array}{ll} 1/r_{\ell} & \mbox{if both $i,j$ belong to clique $\ell$} \\ 0 & \mbox{otherwise.}  \end{array} \right.
$$
By construction, the objective value of (\ref{eqn: k-clique relaxation}) corresponding to $X^*$ is equal to the number of nodes in the $k$-disjoint-clique subgraph
and $\rank(X^*) = k$.
Using the Karush-Kuhn-Tucker conditions, we will derive conditions
for which $X^*$ corresponding to a $k$-disjoint-clique subgraph of $G$, as defined by (\ref{eqn: proposed solution}),
is optimal for the convex relaxation of maximum node $k$-disjoint-clique problem given by (\ref{eqn: k-clique relaxation}).
In particular, these conditions are summarized by the following theorem.

\begin{theorem} \label{thm: KKT conditions}
     Let $X^*$ be feasible for (\ref{eqn: k-clique relaxation}).
     Suppose also that there exists $\blambda \in \R^N_+$, $\mu \in \R$, $\boldeta \in \R^{N\times N}$
     and $S \in \Sigma_+^N$ such that
     \begin{align}
          - \e\e^T + \blambda \e^T +  \e \blambda^T + \mu I + \sum_{\substack{(i,j) \notin E \\ i \neq j}} \boldeta_{ij} \e_i \e_j^T & = S,   \label{dual feas} \\
          \blambda^T (X^* \e - \e) & = 0, \label{CS: row sum} \\
          \langle S, X^* \rangle &= 0.    \label{CS: sdp}
     \end{align}
     Here $\Sigma_+^N$ is the cone of $N \times N$ positive semidefinite matrices,
     $\langle \cdot, \cdot \rangle$ is the trace inner product on $\R^{N\times N}$ defined by
     $$
          \langle Y, Z \rangle = \tr(YZ^T)
     $$
     for all $Y, Z \in \R^{N\times N}$,  $\e_i$ denotes the $i$th column of the identity matrix in $\R^{N\times N}$ for all $i=1,\dots, N$, and $\e$ is the
     all-ones vector in $\R^{N}$.
     Then $X^*$ is an optimal solution of (\ref{eqn: k-clique relaxation}).
\end{theorem}

We omit the proof of this theorem, as it is nothing more than the
specialization of the KKT conditions  \cite{boydvdb} in convex programming to
\eref{eqn: k-clique relaxation}.


\subsection{Construction of the auxiliary matrix $\tilde S$}
\label{sec: notation}
Our proof technique to show that $X^*$ is optimal for
\eref{eqn: proposed solution} is to construct multipliers
to satisfy Theorem~\ref{thm: KKT conditions}.
The difficult multiplier to construct is
$S$, the dual semidefinite matrix.  The reason is that $S$ must
simultaneously satisfy homogeneous linear equations given by $\langle
S,X^*\rangle=0$, requirements on its entries given by the gradient
equation \eref{dual feas}, and positive semidefiniteness.

In this subsection, we will lay the groundwork for our definition
of $S$; in particular, we construct an auxiliary
matrix $\tilde S$.  The actual multipliers
used to prove the optimality of $X^*$, as well as the
proof itself, are in the next subsection.

Our strategy for satisfying the requirements on $S$ is as follows.  The
matrix $S$ will be constructed in blocks with sizes inherited from the
blocks of $X^*$.  In particular, let the nodes contained in the $k$
planted cliques be denoted $C_1,\ldots,C_k$, and let the remaining
nodes be $C_{k+1}$.  Then according to \eref{eqn: proposed solution},
$X^*$ has diagonal blocks $X_{C_q,C_q}^*$ for $q=1,\ldots,k$
consisting of multiples of the all 1's matrix.  The remaining blocks
of $X^*$ are 0's.  The diagonal blocks of $S$ will be
perturbations of the identity matrix, with the rank-one perturbation
chosen so that each diagonal block of $S$, say $S_{C_q,C_q}$ is
orthogonal to $X_{C_q,C_q}$.

The entries of an off-diagonal block, say $S_{C_q,C_s}$ must satisfy, first of all,
\eref{dual feas}.  This constraint, however, is binding only
on the entries corresponding to edges in $G$, since entries
corresponding to absent edges are not constrained by \eref{dual feas}
thanks to the presence of the unbounded multiplier $\eta_{ij}$ on
the left hand side.  These entries that are free in \eref{dual feas}
are chosen so that \eref{CS: sdp} is satisfied.  It is a well known
result in semidefinite programming that the requirements $\langle S,X\rangle=0$,
$X,S\in \Sigma_+^N$ together imply $SX=XS=0$.  Thus, the
remaining entries of $S$ must be chosen so that $X^*S=SX^*=0$.  Because
of the special form of $X^*$, this is equivalent to requiring all
row and column sums of $S_{C_q,C_s}$ to equal zero.

We parametrize the entries of
$S_{C_q,C_s}$ that are not predetermined by \eref{dual feas} using the entries of two
vectors $\y^{q,s}$ and $\z^{q,s}$.  These vectors are chosen to be the
solutions to systems of linear equations, namely, those imposed by the
requirement that $X^*S=SX^*=0$.  We show that the system of linear equations
may be written as a perturbation of a linear system with a known
solution, and we can thus get bounds on $\y^{q,s}$ and $\z^{q,s}$.
The bounds on $\y^{q,s}$ and $\z^{q,s}$ in turn translate to bounds on
$\Vert S_{C_q,C_s}\Vert$, which are necessary to establish the positive
semidefiniteness of $S$.  This semidefiniteness is established by
proving that the diagonal blocks, which are identity matrices plus
rank-one perturbations, dominate the off-diagonal blocks.

Recalling our notation introduced earlier,
$G = (V,E)$ has a $k$-disjoint-clique subgraph $K$ composed of cliques
$C_1, C_2, \dots, C_k$ of sizes $r_1, r_2, \dots, r_k$ respectively.
Let $C_{k+1} := V\setminus (\cup_{i=1}^k C_i)$ be the set of nodes of $G$ not in $K$ and let $r_{k+1} := |C_{k+1}|$.
Let $N := |V|$.
Let $\hat{r} := \min\{r_1, r_2, \dots, r_k\}$.
For each $v \in V$ ,
let $n_{v}^s$ denote the number of nodes adjacent to $v$ in $C_s$ for all $s \in \{1,\dots, k+1\}$,
and let $cl(v)$ denote index $i \in \{1,\dots, k+1\}$ such that $v \in C_{i}$.

Let $A(\bar G) \in \R^{N\times N}$ be the adjacency matrix of the complement $\bar G$ of $G$; that is $[A(\bar G)]_{i,j} = 1$ if $(i,j) \notin E$ and $0$ otherwise.
Next, fix $q, s \in \{1,\dots, k + 1\}$ such that $q \neq s$.
Let $H = H_{q,s} \in \R^{C_q \times C_s}$ be the block of $A(\bar G)$ with entries indexed by the vertex sets $C_q$ and $C_s$,
and let $D = D_{q,s} \in \R^{C_q \times C_q}$ be the diagonal matrix
such that, for each $i \in C_q$, the $(i,i)$th entry of $D$ is equal to the number of nodes in $C_s$ not adjacent to $i$.
That is
$$
     D = r_s I - \Diag(\n_{C_q}^s)
$$
where $\n_{C_q}^s \in \R^{C_q}$ is the vector with $i$th entry equal to $n_i^s$ for each $i \in C_q$.
Let $F = F_{q,s} = D_{s,q}$.
Next, define the scalar
$$
     c = c_{q,s} := \left\{  \begin{array}{ll}
          \displaystyle{\frac{\hat{r}}{2} \left( \frac{1}{r_q} + \frac{1}{r_s} \right),} & \mbox{if } s \le k \\[0.2in]
          \displaystyle{\frac{\hat{r}}{2} \left( \frac{1}{\hat{r}} + \frac{1}{r_q} \right),} & \mbox{otherwise}.
          \end{array}
          \right.
$$
Next, for each $q,s = 1,\dots, k+1$ such that $q\neq s$ let $\b = \b^{q,s} \in \R^{C_q \cup C_s}$ be defined by
$$
     [\b^{q,s}]_i = c \cdot \left \{    \begin{array}{ll}
                                                  n_i^s, & \mbox{if } i \in C_q \\
                                                  n_i^q, & \mbox{if } i \in C_s.
                                             \end{array} \right.
$$
Note that the matrix
$$
     \left( \begin{array}{cc} D & H \\ H^T & F \end{array} \right)
$$
is weakly diagonally dominant since the $i$th row of $H$ contains exactly $r_s-n_i^s$ 1's, and, hence, positive semidefinite.
Further, let $\y = \y^{q,s}$ and $\z = \z^{q,s}$ be a solution of the perturbed system
\begin{equation}\label{eqn: y, z defn 2}
     \left( \begin{array}{cc}
               D + \theta \e\e^T   & H - \theta \e \e^T \\
               H^T - \theta \e \e^T     & F + \theta \e \e^T
          \end{array}
     \right)
     {\y \choose \z} = \b
\end{equation}
for some scalar $\theta > 0$ to be defined later.

The rationale for this
system of equations \eref{eqn: y, z defn 2} is as follows.  Below in \eref{eqn sigma2},
we shall define the matrix $\tilde S_{C_q,C_s}$
according to the formula that entries $(i,j)$
corresponding to edges in $E$ are set to $-c_{q,s}$, while entries
$(i,j)$ corresponding to absent edges are set to the sum
$[\y^{q,s}]_i+[\z^{q,s}]_j$.  Matrix $S_{C_q,C_s}$ has the
same formula; refer to \eref{eqn S cq cs def} below.

As mentioned earlier, it is required that all row and column sums of
$S_{C_q,C_s}$ equal zero.  Consider, e.g., the sum of the entries
in a particular row $i\in C_q$.  This sum consists of $r_s$ terms; of
these terms, $n_i^s$ of them are $-c_{q,s}$ (corresponding to edges
from $i$ to $C_s$) while the other $r_s-n_i^s$ have the form
$[\y^{q,s}]_i+[\z^{q,s}]_j$.  Thus, the
requirement that the row sum to zero is written
$$-n_i^sc_{q,s}+\sum_{j\in C_s; (i,j)\notin E}\left([\y^{q,s}]_i+[\z^{q,s}]_j\right)=0$$
which may be rewritten
\begin{equation}
(r_s-n_i^s)[\y^{q,s}]_i+\sum_{j\in C_s; (i,j)\notin E}[\z^{q,s}]_j = n_i^sc_{q,s}.
\label{eqn row sum}
\end{equation}
Equation \eref{eqn row sum} is exactly  a
row of \eref{eqn: y, z defn 2} in the case $\theta=0$ because of the
formulas used to define $D,F,H,\b$.

In the case that $\theta$ is not zero, the equation for the $i$th
row in \eref{eqn: y, z defn 2}
has an additional
term of the form $\theta(\e^T\y^{q,s}-\e^T\z^{q,s})$.  This additional
term does not affect the result, as the following argument shows.  The
version of \eref{eqn: y, z defn 2} with $\theta=0$ is singular because
the vector $(\e;-\e)$ is in its null space.  
Now fix $\theta  > 0$ such that \eref{eqn: y, z defn 2} is nonsingular.
By the fact that $(\e; -\e)$ is in the null space of the coefficient matrix
$$
	\left( \begin{array}{cc}
               D    & H \\
               H^T     & F 
          \end{array}
     \right),
$$     
taking the inner product of each side of \eqref{eqn: y, z defn 2} with $(\e; -\e)$ yields
$$
	\theta (r_q + r_s) (\e^T \y  -  \e^T\z) = \b_1^T \e - \b_2^T \e,
$$
where $\b_1 \in \R^{C_q}$, $\b_2 \in \R^{C_s}$ are the vectors of entries of $\b$ corresponding to $C_q$ and $C_s$ respectively.
Moreover, 
$$
	\b_1^T \e - \b_2^T \e = \sum_{i\in C_q} n_i^s - \sum_{j\in C_s} n_j^q = 0
$$
because the number of edges entering $C_q$ from $C_s$ is equal to the number of edges entering $C_s$ from $C_q$.
Therefore, if for some $\theta>0$ we are able to show that \eref{eqn: y, z defn 2}
is nonsingular (which we shall establish in Section~\ref{sec: adversarial case}
and again in
Section~\ref{sec: random case}) then this particular
$(\y^{q,s},\z^{q,s})$ satisfying \eref{eqn: y, z defn 2} will also be a
solution to \eref{eqn row sum}  since the additional term
$\theta(\e^T\y^{q,s}-\e^T\z^{q,s})$ is zero.

For the remainder of this section, in order to formulate definitions for the remaining
multipliers, assume that $\theta>0$ and that \eref{eqn: y, z defn 2} is nonsingular.
Furthermore, assume that $D_{ii} > 0$ for all $i \in C_q$ and $F_{ii} > 0$ for all $i\in C_s$.
Let
$$
     A = A(\theta):= \left( \begin{array}{cc}
                    D + \theta \e\e^T & 0 \\
                    0 & F + \theta \e\e^T
               \end{array}
          \right),\;\;\;
     P = P(\theta):=  \left(    \begin{array}{cc}
                         0 & H - \theta \e\e^T \\
                         H^T - \theta \e\e^T & 0
                    \end{array}
          \right);
$$
then we have assumed that $A+P$ is nonsingular, and
$$
     \left( \begin{array}{c} \y \\ \z \end{array} \right) = (A+P)^{-1} \b.
$$
The proof technique in
Sections~\ref{sec: adversarial case}
and \ref{sec: random case}
is to show that $Q:= (A+P)^{-1} - A^{-1}$ is small so that $(\y,\z)$ is close to $A^{-1} \b$.
Let $Q = (Q_1^T, Q_2^T)^T$ where $Q_1 \in \R^{C_q \times (C_q \cup C_s)}$ and $Q_2 \in \R^{C_s \times (C_q \cup C_s)}$.
Then, under this notation,
\begin{align*}
     \left( \begin{array}{c} \y \\ \z \end{array} \right)
          &= A^{-1} \b + {Q_1 \choose Q_2} \b\\
          &= \left( \begin{array}{cc}
                         (D + \theta \e\e^T)^{-1} & 0 \\
                         0 & (F + \theta \e\e^T)^{-1}
                    \end{array}
               \right)
               {\b_1 \choose \b_2} + {Q_1 \choose Q_2}\b.
\end{align*}
Therefore, if $D$, $F$ and $A+P$ are nonsingular,
$$
     \y = (D+\theta \e\e^T)^{-1} \b_1 + Q_1 \b = (I + \theta D^{-1} \e\e^T)^{-1} D^{-1} \b_1 + Q_1 \b
$$
and
$$
     \z = (I + \theta F^{-1} \e\e^T)^{-1} F^{-1} \b_2 + Q_2 \b.
$$
Let $\bar{\y} := \y - Q_1 \b$ and $\bar{\z} := \z - Q_2 \b$.
In order to give explicit formulas for $\bar\y$ and $\bar\z$ we use the well-known Sherman-Morrison-Woodbury formula
(see, for example, \cite[Equation 2.1.4]{Golub-VanLoan:1996}),
stated in the following lemma,  to calculate $(I + \theta D^{-1} \e\e^T)^{-1}$ and $(I + \theta F^{-1} \e\e^T)^{-1}$.

\begin{lemma} \label{Sherman-Morrison formula}
     If $A$ is a nonsingular matrix in $\R^{n\times n}$  and $\u,\v \in \R^{n}$ satisfy $\v^T A^{-1} \u \neq -1$ then
     \begin{equation} \label{SM formula}
          (A + \u\v^T)^{-1} = A^{-1} - \frac{A^{-1} \u \v^T A^{-1}}{1 + \v^T A^{-1} \u}.
     \end{equation}
\end{lemma}

As an immediate corollary of Lemma~\ref{Sherman-Morrison formula}, notice that
\begin{equation} \label{eqn: ybar formula}
     \bar{\y}  = \left( D^{-1} -  \frac{\theta D^{-1} \e\e^T D^{-1}}{1 + \theta \e^T D^{-1} \e} \right) \b_1
               = D^{-1}\left( I - \frac{\theta \e\e^T D^{-1}}{1 + \theta \e^T D^{-1} \e} \right) \b_1.
\end{equation}
and
\begin{equation} \label{eqn: zbar formula}
     \bar{\z}  = \left( F^{-1} - \frac{\theta F^{-1} \e\e^T F^{-1}}{1 + \theta \e^T F^{-1} \e} \right) \b_2
               = F^{-1} \left(I - \frac{\theta \e\e^T F^{-1}}{1 + \theta \e^T F^{-1} \e} \right) \b_2.
\end{equation}

Finally, we define the $(k+1) \times (k+1)$ block matrix $\tilde S \in \R^{N \times N}$ as follows:
\begin{itemize}
     \item[($\tilde \sigma_1$)]
          For all $q \in \{1,\dots, k\}$, let $\tilde S_{C_q, C_q} = 0$.
     \item[($\tilde \sigma_2$)]
          For all $q, s \in \{1,\dots, k\}$ such that $q \neq s$, let
\begin{equation}
               \tilde{S}_{C_q, C_s} = H_{q,s} \circ (\y^{q,s} \e^T + \e (\z^{q,s})^T) + c_{q,s} (H_{q,s} - \e\e^T).
\label{eqn sigma2}
\end{equation}

     \item[($\tilde \sigma_3$)]
          For all $q  \in \{1,\dots, k\}$ and $i \in C_q$, $j \in C_{k+1}$, let
          $$
               [\tilde{S}_{C_q, C_{k+1}}]_{ij} = [\tilde{S}_{C_{k+1}, C_q}]_{ji}
                    = \left\{ \begin{array}{ll}
                                   - c_{q,k+1}, & \mbox{if } (i,j) \in E \\ [0.2 in]
                                   c_{q,k+1}\, n_j^{q}/(r_q - n_j^{q}), & \mbox{otherwise.}
                              \end{array}
                         \right.
          $$
     \item[($\tilde \sigma_4$)]
          Finally, for all $i, j \in C_{k+1}$, choose
          \begin{equation} \label{sigma 4}
               [\tilde{S}_{C_{k+1},C_{k+1}}]_{ij} =
                    \left\{   \begin{array}{ll}
                                   -1, & \mbox{if } (i,j) \in E \mbox{ or } i = j \\
                                   \gamma, & \mbox{if } (i,j) \notin E
                              \end{array}
                    \right.
          \end{equation}
          for some scalar $\gamma$ to be defined later.
\end{itemize}

We make a couple of remarks about ($\tilde\sigma_2$).  As already
noted earlier, this formula defines entries of $\tilde S_{C_q,C_s}$ to
be $-c_{q,s}$ in positions corresponding to edges, and
$[\y^{q,s}]_i+[\z^{q,s}]_j$ in other positions.  The vectors
$\y^{q,s}$ and $\z^{q,s}$ are defined by \eref{eqn: y, z defn 2}
precisely so that, when used in this manner to define $\tilde
S_{C_q,C_s}$, its row and column sums are all 0 (so that $X^*S=SX^*=0$; the
relationship $S_{C_q,C_s}\equiv \tilde S_{C_q,C_s}$ is given by \eref{eqn S cq cs def} below).
The system is square because the number of constraints on $S_{q,s}$
imposed by $X^*S=SX^*=0$ after the predetermined entries are filled in is
$|C_q|+|C_s|$ (one constraint for each row and column), which is the
total number of entries in $\y^{q,s}$ and $\z^{q,s}$.  As mentioned
earlier, there is the slight additional complexity that these
$|C_q|+|C_s|$ equations have a dependence of dimension 1, which
explains why we needed to regularize \eref{eqn: y, z defn 2} with the
addition of the $\theta\e\e^T$ terms.

As a second remark about $\tilde\sigma_2$, we note that $\tilde
S_{C_q,C_s}=\tilde S_{C_s,C_q}^T$.  This is a consequence of our
construction detailed above.  In particular, $\y^{q,s}=\z^{s,q}$,
$H_{q,s}=H_{s,q}^T$, and $D_{q,s}=F_{s,q}$ for all $q,s=1,\ldots,k$
such that $q\ne s$.

\subsection{Definition of the multipliers, optimality and uniqueness}

We finally come to the main theorem of this section, which
provides a sufficient condition for when the $k$-disjoint-clique subgraph of $G$ composed of the
cliques $C_1, \dots, C_k$ is the maximum node $k$-disjoint-clique subgraph of $G$.

\begin{theorem} \label{thm: sufficient condition}
     Suppose that $G = (V,E)$ has a $k$-disjoint-clique subgraph $G^*$ composed of the disjoint cliques $C_1, \dots, C_k$ and
     let $C_{k+1} := V\setminus (\cup_{i=1}^k C_i)$.
     Let $r_i = |C_i|$ for all $i=1,\dots, k+1$, and let $\hat{r} = \min_{i=1,\dots, k} \{r_i\}$.
     Let $X^*$ be the matrix of the form (\ref{eqn: proposed solution})
     corresponding to the $k$-disjoint-clique subgraph generated by $C_1, \dots, C_k$.
     Moreover, suppose that the matrix $\tilde S$ as defined by $(\tilde \sigma_1), \dots, (\tilde \sigma_4)$ satisfies
     \begin{equation}    \label{eqn: suff condn}
          \| \tilde S \| \le \hat{r}.
     \end{equation}
     Then $X^*$ is optimal for (\ref{eqn: k-clique relaxation}), and $G^*$ is the maximum node $k$-disjoint-clique subgraph of $G$.
     Moreover, if $\|\tilde S \| < \hat r$ and
     \begin{equation} \label{eqn: uniqueness cond}
          n_v^q < r_q
     \end{equation}
     for all $v \in V$ and $q \in \{1,\dots, k\} - cl(v)$
     then $X^*$ is the unique optimal solution of \eref{eqn: k-clique relaxation} and
     $G^*$ is the unique maximum node $k$-disjoint-clique subgraph of $G$.
\end{theorem}

\begin{proof}
     We will prove that (\ref{eqn: suff condn}) is a sufficient condition for optimality of $X^*$ by defining multipliers
     $\mu,$ $\blambda$, $\boldeta$, and $S$ and proving that
     if (\ref{eqn: suff condn}) holds then these multipliers satisfy the optimality  conditions given by Theorem~\ref{thm: KKT conditions}.
     Let us define the multipliers $\mu$ and $\blambda$ by
     \begin{align}
          \mu &= \hat{r} = \min \{r_1, r_2, \dots, r_k\},     \label{eqn: alpha def} \\
          \blambda_i &= \frac{(1 - \hat{r} /r_q)}{2}   \;\;\;  \mbox{for all } i \in C_q,     \label{eqn: lambda def}
     \end{align}
     for all $q = 1, \dots, k$ and
     \begin{equation} \label{eqn: noise lambda def}
          \blambda_{i} =  0.
     \end{equation}
     for all $i \in C_{k+1}$.
     Notice that by our choice of $\mu$ and $\blambda$
     we have
     $$
          S_{C_q,C_q} = \hat{r} I - (\hat{r} /r_q) \e\e^T
     $$
     for all $q = 1,\dots, k$ by (\ref{dual feas}).
     Moreover, we choose
     $\boldeta$ such that
     $$
          \boldeta_{ij} = \left\{   \begin{array}{ll}
                                        \tilde{S}_{ij} - \blambda_i - \blambda_j + 1, &\mbox{if } (i,j) \notin E, i \neq j \\
                                        0,   & \mbox{otherwise}
                                   \end{array}
                        \right.
     $$
     for all $i,j \in V$.
     Note that, by our choice of $\boldeta$, we have
	\begin{equation}
	          S_{C_q, C_s} = \left\{   \begin{array}{ll}
	                                        \tilde{S}_{C_q, C_s},     & \mbox{if } q,s \in \{1,\dots, k+1\}, \; q \neq s \\
	                                        \tilde{S}_{C_{k+1}, C_{k+1}} + \hat{r} I, & \mbox{if } q=s=k+1.
	                                   \end{array}
	                        \right.
	\label{eqn S cq cs def}
	\end{equation}
     by \eref{dual feas}.

     By construction, $\mu, \blambda, \boldeta,$ and $S$ satisfy (\ref{dual feas}).
     Since the $i$th row sum of $X^*$ is equal to $1$ for all $i \in C_q$ for all $q=1,\dots k$ and is equal to $0$ for all $i \in C_{k+1}$,
     $X^*$ and $\blambda$ satisfy the complementary slackness condition (\ref{CS: row sum}).
     Moreover,
     \begin{align*}
          \langle X^*, S \rangle = \sum_{q = 1}^k \left(\sum_{i\in C_q} \sum_{j \in C_q} \frac{1}{r_q} [S_{C_q, C_q}]_{i,j}\right)
               = \sum_{q=1}^k \frac{1}{r_q} \left( r_q \hat r - r_q^2 \left( \frac{\hat r}{r_q} \right) \right) = 0,
     \end{align*}
     and thus $X^*$ and $S$ satisfy (\ref{CS: sdp}).
     It remains to prove that (\ref{eqn: suff condn}) implies that $S$ is positive semidefinite.

     To prove that $S$ is positive semidefinite we show that $\x^T S \x \ge 0$ for all $\x \in \R^N$ if $\tilde S$
     satisfies (\ref{eqn: suff condn}).
     Fix $\x \in \R^N$ and decompose $\x$ as $\x = \x_1 + \x_2$ where
     $$
          \x_1(C_i) = \left\{ \begin{array}{ll}
                                   \boldsymbol{\phi}_i\, \e, & i \in \{1,\dots, k\} \\
                                   0, & i = k+1
                              \end{array}
                         \right.
     $$
     for $\boldsymbol{\phi} \in \R^{k}$ chosen so that
     $\x_2(C_i)^T \e = 0$ for $i=1\dots,k$, $\x_2(C_{k+1}) = \x(C_{k+1})$.
     Here,  $\x(C_i)$, $\x_1(C_i)$, and $\x_2(C_i)$ denote the vectors in $\R^{C_i}$ composed of the entries indexed by $C_i$ of $\x$, $\x_1$, and $\x_2$ respectively, for each $i=1,\dots, k+1$.
     Note that $\x_1$ is in the column space of $X^*$.
     Then $\x_1$ is in the null space of $S$ and  we have
     \begin{align*}
          \x^T S \x &= \x_2^T S \x_2 
               = \hat{r} \|\x_2\|^2 + \x_2^T \tilde{S} \x_2 
               \ge (\hat{r} - \|\tilde S\|)\|\x_2\|^2,
     \end{align*}
     since $\x_2(C_i)$ is orthogonal to $\e$ for all $i=1,\dots,k.$
     Therefore, $S$ is positive semidefinite, and, hence, $X^*$ is optimal for \eref{eqn: k-clique relaxation}
     if $\|\tilde S\| \le \hat{r}$.

     Now suppose that $\|\tilde S\| < \hat r$ and, for all $i = 1,\dots, k$, no node in $C_i$ is adjacent to every node in some other clique.
     Then $X^*$ is optimal for \eref{eqn: k-clique relaxation}.
     For all $i=1,\dots, k,$ let $\v_i \in \R^N$ be the characteristic vector of $C_i$.
     That is,
     $$
          [\v_i]_{j} = \left\{ \begin{array}{ll} 1, &\mbox{if } j \in C_i \\ 0, & \mbox{otherwise.} \end{array} \right.
     $$
     Notice that $X^* = \sum_{i=1}^k (1/r_i) \v_i (\v_i)^T$.
     Moreover, by complementary slackness, $\langle X^*, S \rangle = 0$ and, thus, $\v_i$ is in the null space of $S$ for all $i =1,\dots, k$.
     On the other hand, consider nonzero $\x \in \R^N$ such that $\x^T \v_i = 0$ for all $i=1,\dots, k$.
     That is, $\x$ is orthogonal to the span of $\{\v_i: i =1,\dots,k \}$.
     Then
     \begin{align*}
          \x^T S \x &= \hat{r} \|\x \|^2 + \x^T \tilde S \x 
                     \ge (\hat r  - \|\tilde S \|) \|\x\|^2 > 0.
     \end{align*}
     Therefore, $\mathrm{Null}(S) = \mathrm{span} \{\v_i: i =1,\dots,k \}$ and $\rank(S) = N-k$.

     Now suppose that $\hat X$ is also optimal for \eref{eqn: k-clique relaxation}.
     Then, by complementary slackness, $\langle \hat X , S \rangle = 0$ which holds if and only if $\hat X  S = 0$.
     Therefore, the row and column spaces of $\hat X$ lie in the null space of $S$.
     It follows immediately, since $\hat X \succeq 0$, that $\hat X$ can be written in the form
     $$
          \hat X = \sum_{i=1}^k \sigma_i \v_i \v_i^T + \sum_{i=1}^k \sum_{\substack{j = 1 \\ j \neq i}}^{k} \omega_{i,j} \v_i \v_j^T
     $$
     for some $\sigma \in \R^k_+$ and $\omega \in \Sigma^k$, where $\Sigma^k$ denotes the set of $k\times k$ symmetric matrices.
     Now, if $\omega_{i,j} \neq 0$ for some $i \neq j$ then every entry in the block $\hat X(C_i, C_j) = \hat X(C_j,C_i)^T$ must be equal to $\omega_{i,j}$.
     Since each of these entries is nonzero, this implies that each node in $C_i$ is adjacent to every node in $C_j$,
     contradicting Assumption~\eref{eqn: uniqueness cond}.
     Therefore, $\hat X$ has singular value decomposition $\hat X = \sigma_1 \v_1 \v_1^T + \dots + \sigma_k \v_k \v_k^T$.
     Moreover, since $\hat X$ is optimal for \eref{eqn: k-clique relaxation} it must have objective value equal to that of $X^*$ and thus
     \begin{equation} \label{uniqueness proof 1}
          \sum^k_{i=1} r_i = \sum_{i=1}^N \sum_{j=1}^N X^*_{i,j} = \sum_{i=1}^N \sum_{j=1}^N \hat X_{i,j} = \sum_{i=1}^k \sigma_i r_i^2.
     \end{equation}
     Further, since $\hat X$ is feasible for \eref{eqn: k-clique relaxation},
     \begin{equation} \label{uniqueness proof 2}
          \sigma_i r_i \le 1
     \end{equation}
     for all $i = 1,\dots, k$.
     Combining \eref{uniqueness proof 1} and \eref{uniqueness proof 2} shows that $\sigma_i = 1/r_i$ for all $i=1,\dots, k$ and, hence, $\hat X = X^*$ as required.
     \qed
     \bigskip
\end{proof}

\section{The Adversarial Case}
\label{sec: adversarial case}

Suppose that the graph $G= (V,E)$ is generated as follows. We first add $k$ disjoint cliques with vertex sets
$C_1, \dots, C_k$ of size $r_1, r_2, \dots, r_k$ respectively.
Then, an adversary is allowed to add a set $C_{k+1}$ of additional vertices and a number of the remaining
potential edges to the graph. We will show that our adversary can add up to $O(\hat{r}^2)$ noise edges
where $\hat{r} = \min\{r_1, \dots, r_k\}$ and the $k$-disjoint-clique subgraph composed of
$C_1, \dots, C_k$ will still be the unique maximum $k$-disjoint-clique subgraph of $G$.

The main theorem concerning the adversarial case is as follows.

\begin{theorem}
\label{thm: adv case}
Consider an instance of the $k$-disjoint-clique problem constructed according
to the preceding description, namely,
$G$ contains a $k$-disjoint-clique graph $G^*$
whose nodes are partitioned as  $C_1\cup\cdots\cup C_k$  where  $|C_q|=r_q$, $q=1,\ldots,k$,
plus additional nodes denoted
$C_{k+1}$ and additional edges (which may have
endpoints chosen from any of $C_1,\ldots,C_{k+1}$).
Let
$\hat r=\min(r_1,\ldots,r_k)$.
Assume the following conditions are satisfied:
\begin{enumerate}
\item
For all $q=1,\ldots, k$, $i\in C_q$, for all $s\in \{1,\ldots,k+1\}-q$,
\begin{equation}
\label{adjacency cond}
n_i^s \le \delta\min(r_q,r_s).
\end{equation}
Here, $\delta$ is a scalar satisfying \eref{eqn: ac delta range} below.
\item $|E(G)\setminus E(G^*)| \le \rho \hat r^2,$
where $\rho$ is a positive scalar depending on $\delta$.
\end{enumerate}
 Then
$X^*$ given by \eref{eqn: proposed solution} is the unique optimal solution to
\eref{eqn: k-clique relaxation}, and $G^*$ is the unique optimal
solution of the $k$-disjoint clique problem.
\end{theorem}

We remark that two of the conditions imposed in this theorem are, up to the constant
factors, the best possible according to the following information-theoretic
arguments.  First, if $n_i^s=r_s$, then node $i$ could be inserted into clique $s$, so
the partitioning between $C_s$ and $C_q$ would no longer be uniquely determined.
This shows the necessity of the condition $n_i^s\le O(r_s)$.
The condition that $|E(G)\setminus E(G^*)|\le \rho\hat r^2$ is necessary, up
to the constant factor, because
if $|E(G)\setminus E(G^*)|\ge\hat r (\hat r -1)/2$, then we could interconnect an
arbitrary set of $\hat r$ nodes chosen from among the existing cliques with edges
to make
a new clique out of them, again spoiling the uniqueness of the decomposition.

An argument for the necessity of the condition that $n_i^s\le \delta r_q$ is not
apparent, so possibly there is a strengthening of this theorem that drops
that condition.

The remainder of this section is the proof of this theorem.  As might
be expected, the proof hinges on establishing \eref{eqn: suff condn};
once this inequality is established, then
Theorem~\ref{thm: sufficient condition} completes the argument.

As before, let $r_{k+1}$ denote $|C_{k+1}|$.
For the remainder of the proof, to simplify the notation, we assume
that $r_{k+1}\le 2\rho \hat r^2$.  The reason is that since
$|E(G)\setminus E(G^*)| \le \rho \hat r^2$ by assumption, if
$r_{k+1}> 2\rho \hat r^2$ then $C_{k+1}$ would include one or more isolated
nodes (i.e, nodes of degree 0), and these nodes can simply be deleted in a
preliminary phase of the algorithm.  (The algorithm still works with an
arbitrary number of isolated nodes in $G\setminus G^*$, but the notation in the
proof requires some needless additional complexity.)

Recall that the construction of the multipliers presented in Section~\ref{sec:k-clique} depended on two
scalars $\theta$ in \eref{eqn: y, z defn 2} and $\gamma$ in \eref{sigma 4}:
choose $\theta = 1$ and $\gamma = 0$.

We impose the assumption that $\delta\in(0,0.382)$.
The constant $0.382$ is chosen so that
\begin{equation} \label{eqn: ac delta range}
    0 < \delta < (1-\delta)^2.
\end{equation}
We will show that, under this assumption, there exists some $\beta > 0$ depending only on $\delta$ such that
$$
    \| \tilde{S}_{C_q, C_s} \|^2 \le \beta \|\b^{q,s}\|_1
$$
for all $q,s \in \{1,\dots, k\}$ such that $q\neq s$.

Choose $q,s \in \{1,\dots, k\}$ such that $q \neq s$ and let
$D$, $F$,  $H$, $\b,$ and $c$  be defined as in Section~\ref{sec: notation}.
Without loss of generality we may assume that $r_q \le r_s$.
Moreover, let $\y$ and $\z$ be the solution of the system \eref{eqn: y, z defn 2} and define $A, Q, P$ as in Section~\ref{sec: notation}.

We begin by showing that, under this assumption, $\y$ and $\z$ are uniquely determined.
Note that, since $n_i^s = r_s - D_{ii} \le \delta r_s$ for all $i \in C_q$ and $n_i^q = r_q - F_{ii} \le \delta  r_q$ for all $i\in C_s$ by Assumption~\eref{adjacency cond},
$D$ and $F$ are nonsingular and hence $A$ is nonsingular.
Moreover,
$$
     A + P = A (I + A^{-1} P)
$$
and, hence, $A + P$ is nonsingular if $\|A^{-1} P\| < 1$.
Note that, for all $t > 0$, we have
\begin{equation} \label{eqn: min eig bound perturb}
     \lambda_{\min}(D + t \e\e^T)  \ge \lambda_{\min}(D) = \min_{i \in C_q} D_{ii}
\end{equation}
since $\e\e^T \succeq 0$ where $\lambda_{\min}(D + t \e\e^T)$ is the smallest eigenvalue of the symmetric matrix $D + t \e\e^T$.
Taking $t = 1$ in \eref{eqn: min eig bound perturb} shows that
\begin{equation} \label{eqn: A bound 1 ac}
     \|(D + \e\e^T)^{-1}\| \le \|D^{-1}\| = \frac{1}{\min_{i \in C_q} D_{ii}} \le \frac{1}{(1- \delta) r_s}
\end{equation}
since, for each $i \in C_q$, we have
$$
     (1- \delta) r_s \le  D_{ii} \le r_s
$$
by Assumption~\eref{adjacency cond}.
Similarly,
\begin{equation} \label{eqn: A bound 2 ac}
     \| (F + \e\e^T)^{-1}\| \le \|F^{-1}\| = \frac{1}{\min_{j \in C_s} F_{ii}} \le \frac{1}{(1-\delta) r_q}.
\end{equation}
Combining \eref{eqn: A bound 1 ac} and \eref{eqn: A bound 2 ac} we have
\begin{align}
     \|A^{-1}\| &= {\max \{\|(D+\e\e^T)^{-1}\|, \|(F+\e\e^T)^{-1}\| \} } \notag \\ &
               \le \frac{1}{(1-\delta) \min\{r_q, r_s\}} \notag \\ &
               = \frac{1}{(1-\delta) r_q}. \label{eqn: A bound ac}
\end{align}
On the other hand,
\begin{equation} \label{eqn: P bound ac}
     \|P \| = \|H - \e\e^T\| \le \|H - \e\e^T\|_F = \left( \sum_{i\in C_q} \sum_{j \in C_s} (H_{ij} - 1)^2 \right)^{1/2} \le \sqrt{\delta} r_q
\end{equation}
since $H_{ij} - 1$ is equal to $-1$ in the case that $(i,j) \in E$ and $0$ otherwise and there at most $\delta r_q^2$ edges between $C_q$ and $C_s$
by Assumption~\eref{adjacency cond}.
Therefore, since $\delta < (1-\delta)^2$ by Assumption (\ref{eqn: ac delta range}), we have
$$
     \|A^{-1} P\| \le \|A^{-1}\| \|P \| \le \frac{\sqrt{\delta}}{1-\delta} < 1
$$
and, thus, $A + P$ is nonsingular and $\y$ and $\z$ are uniquely determined.

Now, recall that
$$
     \tilde{S}_{C_q, C_s} = H \circ (\y \e^T + \e \z^T) -c (\e\e^T - H).
$$
In order to calculate an upper bound on $\|\tilde{S}_{C_q, C_s}\|$ we write $\tilde{S}_{C_q, C_s}$ as
\begin{equation}    \label{eqn: M decomp ac}
     \tilde{S}_{C_q, C_s} = m_1 + m_2 + m_3 + m_4 + m_5
\end{equation}
where
\begin{equation}    \label{eqn: mi defn ac}
    \begin{array}{c}
     m_1 := H \circ (\bar{\y} \e^T), \;\;\; m_2 := H \circ (\e \bar{\z}^T), \;\;\; m_3:= H \circ (Q_1 \b \e^T), \\
     m_4:= H \circ (\e (Q_2 \b)^T), \;\;\; m_5 := -c (\e \e^T - H)
    \end{array}
\end{equation}
and apply the triangle inequality to obtain
\begin{equation}
     \| \tilde{S}_{C_q, C_s} \| \le \sum_{i=1}^5 \| m_i \|.
\end{equation}
Throughout our analysis of $\|\tilde{S}_{C_q, C_s}\|$ we will use the following series of inequalities.
For any $W \in \R^{m\times n}$, $\u \in \R^m$ and $v \in \R^n$, we have
\begin{align}
     \|W \circ \u \v^T\| = \|\Diag(\u) \cdot W \cdot \Diag(\v) \| &\le \|\Diag(\u)\| \|\Diag(\v)\| \|W\| \notag \\
          &= \|\u \|_\infty \|\v\|_\infty \|W\|. \label{chain 0}
\end{align}
On the other hand,
\begin{align}
     \|W \circ \u \v^T\| &= \|\Diag(\u) \cdot W \cdot \Diag(\v)\| \notag \\
          &\le \|\v\|_\infty \|\Diag(\u)\cdot W \| \le \|\v\|_\infty \|\Diag(\u) \cdot W \|_F\notag \\
          &= \|\v\|_\infty \left( \sum_{i=1}^m \u_i^2\|W(i,:)\|^2 \right)^{1/2} \label{chain 1} \\
          &\le \|\u\| \|\v \|_\infty \max_{i =1,\dots, m} \|W(i,:)\| \label{chain 2}
\end{align}
and
\begin{equation}
     \|W \circ \u \v^T\| \le \|\u\|_\infty \|\v\| \max_{j =1,\dots, n} \|W(:, j)\| \label{chain 3}
\end{equation}
where $W(i,:)$ and $W(:,j)$ denote the $i$th and $j$th row and column of $W$.

We begin with $\|m_1\|$.
Applying the bound \eref{chain 2} with $W = H$, $\u =\bar \y$, and $\v = \e$ we have
\begin{equation} \label{eqn: m1 upper bound 1 ac}
     \|m_1\|^2 \le \max_{i \in C_q} D_{ii} \|\bar \y \|^2.
\end{equation}
Here, we used the fact that $\max_{i\in C_q} \Vert H(i,:)\Vert=\max_{i\in C_q}D_{ii}^{1/2}$ since the $i$th row
of $H$ contains exactly $r_s-n_i^s$ 1's.
Thus, since
\begin{align*}
     \| \bar \y \| \le \|(D + \e\e^T)^{-1}\| \| \b_1 \| \le \frac{ \| \b_1\|}{\min_{i \in C_q}{D_{ii}}}
          \le \frac{\| \b_1\|}{(1- \delta) r_s},
\end{align*}
it follows immediately that
\begin{equation} \label{eqn: m1 bound ac}
     \|m_1\|^2 \le \frac{1}{(1-\delta)^2 r_s} \|\b_1\|^2
\end{equation}
since $D_{ii} \le r_s$ for all $i\in C_q$.
By an identical calculation, we have
\begin{equation} \label{eqn: m2 bound ac}
     \|m_2\|^2 \le \frac{1}{(1-\delta)^2 r_q} \|\b_2\|^2.
\end{equation}
Next, applying \eref{chain 2} with $W = H$, $\u = Q_1 \b$, $\v = \e$ yields
\begin{align*}
     \|m_3\|^2 \le \max_{i \in C_q} D_{ii}\|Q_1 \b\|^2 \le r_s \|Q_1 \b\|^2 \le r_s \| Q_1 \|^2 \| \b \|^2
\end{align*}
since $\max_i D_{ii} \le r_s$.
To derive an upper bound on  $\|m_3\|^2$ we first derive an upper bound on $\|Q_1\|$.

Note that
\begin{equation} \label{eqn: Q expansion}
     Q = (A + P)^{-1} - A^{-1} = (( I + A^{-1}P)^{-1} - I) A^{-1} = \sum_{\ell=1}^\infty (-A^{-1} P)^\ell A^{-1}
\end{equation}
since $(I + X)^{-1} = \sum_{\ell = 0}^\infty (-X)^\ell$ for all $X$ such that $\|X\| < 1$ by Taylor's Theorem.
Notice that
     $$
          A^{-1} P = \left( \begin{array}{cc} 0 & P_1 \\ P_2 & 0 \end{array} \right)
     $$
     where
     \begin{align*}
          P_1 &= (D+\theta \e\e^T)^{-1} (H - \theta \e\e^T), \;\;\;\;
          P_2 = (F +\theta \e\e^T)^{-1} (H^T - \theta \e\e^T).
     \end{align*}
     It follows immediately that
     \begin{align}
          Q   &= \sum_{\ell=0}^\infty \left(
               \left( \begin{array}{cc} (P_1 P_2)^{\ell + 1} & 0 \\ 0 & (P_2 P_1)^{\ell + 1} \end{array} \right) -
               \left( \begin{array}{cc} 0 & (P_1 P_2)^\ell P_1 \\ (P_2 P_1)^\ell P_2 & 0 \end{array} \right)
               \right) A^{-1} \label{eqn: Q full expansion}
     \end{align}
     since, for any integer $\ell \ge 1$
     $$
          \left( \begin{array}{cc} 0 & P_1 \\ P_2 & 0 \end{array} \right)^\ell
               =\left\{ \begin{array}{ll} \left( \begin{array}{cc} (P_1 P_2)^{\ell/2} & 0 \\ 0 & (P_2 P_1)^{\ell/2} \end{array} \right), & \mbox{if $\ell$ even} \\[0.2in]
                                   \left(\begin{array}{cc} 0 & (P_1 P_2)^{(\ell -1)/2} P_1 \\ (P_2 P_1)^{(\ell - 1)/2} P_2 & 0 \end{array} \right), &\mbox{if $\ell$ odd.}
                    \end{array}
               \right.
     $$
     Therefore,
     \begin{equation} \label{eqn: Q1 bound 1}
          \|Q_1\| \le \|(D+\theta \e\e^T)^{-1}\| \sum_{\ell = 1}^\infty \|P_1 P_2\|^\ell
               + \|P_1\|\|(F + \theta \e\e^T)^{-1}\| \sum_{\ell = 0}^\infty  \|P_1 P_2\|^\ell
     \end{equation}
     and
     \begin{equation} \label{eqn: Q2 bound 1}
          \|Q_2\| \le \|(F+\theta \e\e^T)^{-1}\| \sum_{\ell = 1}^\infty \|P_1 P_2\|^\ell
               + \|P_2\|\|(D + \theta \e\e^T)^{-1}\| \sum_{\ell = 0}^\infty  \|P_1 P_2\|^\ell.
     \end{equation}
Substituting \eref{eqn: A bound 1 ac}, \eref{eqn: A bound 2 ac} and \eref{eqn: P bound ac} into (\ref{eqn: Q1 bound 1}) yields
\begin{align}
    \|Q_1\| &\le \frac{1}{(1-\delta) r_s} \sum^\infty_{\ell=1} \left( \frac{\delta}{(1-\delta)^2} \right)^\ell
     + \frac{\delta^{1/2}}{(1-\delta)^2 r_s} \sum^\infty_{\ell = 0} \left( \frac{\delta}{(1-\delta)^2} \right)^\ell  
      \le \tilde c /r_s  \label{eqn: Q bound ac}
\end{align}
where
$$
     \tilde{c} =  \frac{2\max \{ \delta/(1-\delta), \sqrt{\delta} \}} {(1-\delta)^2 - \delta}
$$
since
$$
     \|P_1 P_2\| \le \|H - \e\e^T\|^2\|D^{-1}\| \|F^{-1}\| \le \frac{\delta}{(1-\delta)^2}.
$$
Note that  Assumption (\ref{eqn: ac delta range}) ensures that  the infinite series in (\ref{eqn: Q bound ac}) converge.
It follows that
\begin{equation} \label{eqn: m3 bound ac}
     \|m_3\|^2 \le \frac{\tilde c^2}{r_q} \|\b\|^2 .
\end{equation}
On the other hand,
\begin{align}
     \|Q_2\| &\le \frac{1}{(1-\delta)r_q} \sum^\infty_{\ell=1} \left( \frac{\delta}{(1-\delta)^2} \right)^\ell
          + \frac{\sqrt{\delta}}{(1-\delta)^2  r_s} \sum^\infty_{\ell=0} \left( \frac{\delta}{(1-\delta)^2} \right)^\ell \notag \\
          &\le \tilde{c} / r_q \label{eqn: Q2 bound ac}
\end{align}
since $r_s \ge r_q$.
Thus, applying \eref{chain 3} with $W = H$, $\u = \e$, $\v = Q_2 \b$ we have
\begin{equation}    \label{eqn: m4 bound}
     \|m_4\|^2 \le r_q \|Q_2\|^2 \|\b\|^2 \le \frac{\tilde c^2}{r_q} \|\b\|^2.
\end{equation}
Finally,
\begin{align}
     \|m_5\|^2 = \|c(H - \e\e^T)\|^2 &\le \|c(H - \e\e^T)\|_F^2 \notag \\
          & = c\sum_{i \in C_q} \sum_{j \in C_s} (H_{ij} - 1)^2 \notag\\
          &= c\sum_{i\in C_q} n_i^s = \|\b_1\|_1. \label{eqn: m5 bound}
\end{align}
Therefore, there exists $\beta \in \R$ such that
\begin{equation} \label{eqn: M bound ac a}
     \|\tilde S_{C_q, C_s}\|^2 \le \beta \frac{\|\b\|^2}{r_q} + \|\b\|_1
\end{equation}
where $\beta$ depends only on $\delta$.
Moreover, since $\|\b\|^2 \le \|\b\|_1 \|\b\|_\infty$ and
\begin{equation} \label{eqn: b infty norm ac}
     \|\b\|_\infty  = c \cdot \max \left\{ \max_{i\in C_q} n_i^s, \max_{i \in C_s} n_i^q \right\} \le \delta c \min\{r_q, r_s\}
                    = \delta c r_q
\end{equation}
by Assumption \eref{adjacency cond}, there exists $\tilde \beta$ depending only on $\delta$ such that
\begin{equation} \label{eqn: M bound ac}
     \|\tilde S_{C_q, C_s}\|^2 \le \tilde \beta \|\b\|_1
\end{equation}
as required.

Next, consider $\tilde S_{C_q, C_{k+1}}$ for some $q \in \{1,\dots, k\}$.
Recall that
$$
     [\tilde S_{C_q, C_{k+1}}]_{ij} = \left\{  \begin{array}{ll}
                                                       -c,        & \mbox{if }(i,j) \in E \\
                                                       {c n_j }/{(r_q - n_j)},         &\mbox{otherwise}
                                                  \end{array}
                                        \right.
$$
where $n_j = n_j^{q}$ is the number of edges from $j \in C_{k+1}$ to $C_q$ for each $j \in C_{k+1}$.
Hence,
\begin{align}
     \| \tilde S_{C_q, C_{k+1}}\|^2 &\le \| \tilde S_{C_q, C_{k+1}} \|^2_F \notag\\
          &= \sum_{j \in C_{k+1}} \left( n_j c^2 + (r_q - n_j) \left(\frac{n_j c}{r_q - n_j} \right)^2 \right) \notag\\
          &\le c^2 \sum_{j \in C_{k+1}} \left(n_j + \frac{\delta n_j}{(1-\delta)} \right) \notag\\
          &= \frac{c^2}{1-\delta} |E(C_q, C_{k+1})| \label{eqn:S cq ck+1}
\end{align}
where $E(C_q, C_{k+1})$ is the set of edges from $C_q$ to $C_{k+1}$.
Similarly, by our choice of $\gamma = 0$ in $(\tilde \sigma_4)$, we have
\begin{align}
     \|\tilde S_{C_{k+1}, C_{k+1}}\|^2 &
          \le \|\tilde S_{C_{k+1}, C_{k+1}} \|_F^2 
          	= r_{k+1} + 2 |E(C_{k+1}, C_{k+1})|.
\label{eqn:S ck+1 ck+1}
\end{align}
Let $B$ be the vector obtained by concatenating $\b^{q,s}$ for all $q,s \in \{1, \dots, k\}.$
Then, there exist scalars $\hat{c}_1,$ $\hat{c}_2\in \R$ depending only on $\delta$ such that
\begin{align*}
    \sum_{q=1}^{k+1} \sum_{s=1}^{k+1} \|\tilde S_{C_q, C_s}\|^2 
       &= \sum_{\substack{q,s \in \{1,\dots, k\} \\ q\neq s}} \|\tilde S_{C_q, C_s}\|^2 + 2 \sum_{q =1}^k \|\tilde S_{C_q, C_{k+1}}\|^2
                    + \|\tilde S_{C_{k+1}, C_{k+1}} \|^2 \\
          &\le \hat{c}_1 \|B \|_1  + \hat{c}_2 \sum_{q=1}^{k+1} |E(C_{q}, C_{k+1})| + r_{k+1}
\end{align*}
by \eref{eqn: M bound ac}, \eref{eqn:S cq ck+1} and \eref{eqn:S ck+1 ck+1}.
It follows that, since $\|\b^{q,s}\|_1 \le  2|E(C_q, C_s)|$ for all $q, s \in \{1,\dots, k\}$ such that $q\neq s$, there exists
$ \hat{c}_3 \ge 0$ depending only on $\delta$ such that
$$
     \sum_{q=1}^{k+1} \sum_{s=1}^{k+1} \|\tilde S_{C_q, C_s}\|^2 \le \hat c_3 R + r_{k+1}
$$
where $R := |E(G) \setminus E(G^*)|$ is the number of edges of $G$ not contained in the $k$-disjoint-clique subgraph $G^*$ composed of $C_1, \dots, C_k$.
The hypothesis of the theorem is that $R\le \rho\hat r^2$.  We have also assumed earlier that
$r_{k+1}\le 2\rho\hat r^2$.
Hence the sum of the squares of the $2$-norms of the blocks of $\tilde S$ is at most $(\hat c_3+2)\rho\hat{r}^2$.
Therefore, there exists some $\rho > 0$ depending only on $\delta$ such that the preceding inequality implies
$\| \tilde S \| \le \hat{r}$.  This proves the theorem.

\section{The Randomized Case}
\label{sec: random case}

Let $C_1, C_2, \dots, C_{k+1}$ be disjoint vertex sets of sizes $r_1, \dots, r_{k+1}$ respectively, and let $V = \cup_{i=1}^{k+1} C_i$.
We construct the edge set of the graph $G = (V,E)$ as follows:
\begin{itemize}
     \item[($\Omega_1$)]
          For each $q = 1,\dots, k$, and each $i \in C_q,$ $j \in C_q$ such that $i\neq j$ we add $(i,j)$ to $E$.
     \item[($\Omega_2$)]
          Each of the remaining possible edges is added to $E$ independently at random with probability $p \in (0,1)$.
\end{itemize}
Notice that, by our construction of $E$, the graph $G = (V,E)$ has a $k$-disjoint-clique subgraph $G^*$ with cliques indexed
by the vertex sets  $C_1, \dots, C_k$.
We wish to determine which random graphs $G$ generated according to ($\Omega_1$)
and ($\Omega_2$) have maximum node $k$-disjoint-clique subgraph equal to $G^*$  and can be found
with high probability via solving \eref{eqn: k-clique relaxation}.
We begin by providing a few results concerning random matrices
with independently identically distributed (i.i.d.) entries of
mean 0.

\subsection{Results on norms of random matrices}
\label{sec: random matrix theory}

Consider the probability distribution $\P$ for a random variable $x$ defined as follows:
$$
     x =\left\{
     \begin{array}{ll}
          1 & \mbox{with probability $p$,} \\
          -p/(1-p) & \mbox{with probability $1-p$.}
     \end{array}
\right.
$$
It is easy to check that the mean of $x$ is 0 and the variance of $x$ is $\sigma^2=p/(1-p)$.
In this section we provide a few results concerning random matrices with entries independently identically distributed (i.i.d.) according to $\P$.
We first recall a theorem of Geman \cite{Geman:1980} which provides a bound on the largest singular value of a random
matrix with independently identically distributed (i.i.d.) entries of mean $0$.

\begin{theorem}
    \label{geman tail bound}
    Let $A$ be a $\lceil yn \rceil \times n$ matrix whose entries are chosen according to $\P$
    for fixed $y \in \R_+$.
    Then, with probability at least $1 - c_1 \exp(-c_2 n^{c_3})$
    where $c_1 > 0$, $c_2>0$, and $c_3 > 0$ are constants depending on $p$ and $y$,
    $$
        \|A\| \le c_4 \sqrt{n}
    $$
    for some constant $c_4$.
\end{theorem}

Note that this theorem is not stated exactly in this form in \cite{Geman:1980}, but can be
deduced by taking $k= n^q$ for a $q$ satisfying $(2\alpha + 4)q < 1$ in the equations on pp.~255-256.
A similar theorem due to F\"{u}redi and Koml\'{o}s \cite{Furedi-Komlos:1981} exists for symmetric matrices $A$ with
entries distributed according to $\P$.

\begin{theorem}
    \label{Furedi-Komlos}
        For all integers $i,j$, $1\le j \le i \le n$, let $A_{ij}$ be distributed according to $\P$.
        Define symmetrically $A_{ij} = A_{ji}$ for all $i < j$.

        Then the random symmetric matrix $A = [A_{ij}]$ satisfies
        \[
            \| A \| \le 3 \sigma \sqrt{n}
    \]
    with probability at least to $1 - \exp(-c n^{1/6})$ for some
$c > 0$ that depends on $\sigma$.
\end{theorem}

As in Theorem~\ref{geman tail bound}, the theorem is not stated exactly in this manner
in \cite{Furedi-Komlos:1981}; the stated form of the theorem can be deduced by taking $k = (\sigma / K)^{1/3} n^{1/6}$ and $v = \sigma \sqrt{n}$ in
the inequality
$$
    P(\max |\lambda| > 2\sigma \sqrt{n} + v) < \sqrt{n} \exp(-kv /(2 \sqrt{n} + v))
$$
on p.~237.

Next, we state a version of the well-known Chernoff bounds which provides a  bound on the tail distribution
of a sum of independent Bernoulli trials (see \cite[Theorems 4.4 and 4.5]{Mitzenmacher-Upfal:2005}).

\begin{theorem}[Chernoff Bounds]
    \label{CH ineq}
        Let $X_1, \dots, X_k$ be a sequence of $k$ independent Bernoulli trials, each succeeding with probability $p$
        so that $E(X_i) = p$.
        Let $S = \sum_{i=1}^k X_i$ be the binomially distributed variable describing the total number of successes.
        Then for $\delta > 0$
        \begin{equation}
            \label{Chernoff bound}
            P \Big(S > (1+\delta)pk \Big) \le \left(\frac{e^{\delta}}{(1+\delta)^{(1+\delta)}}\right)^{pk}.
        \end{equation}
It follows that for all $a \in (0, p \sqrt{k})$,
\begin{equation}
    \label{Chernoff bound sqrt root}
    P(|S - pk| > a \sqrt{k}) \le 2 \exp(-a^2/p).
\end{equation}
\end{theorem}

The final theorem of this section is as follows (see \cite[Theorem 4]{Ames-Vavasis}).
\begin{theorem} \label{thm: 2.4}
Let $A$ be an $n\times N$ matrix whose entries are chosen according to
$\P$. Suppose also that $\log N\le \sqrt{n}$.
Let $\tilde A$ be defined as follows.  For $(i,j)$
such that $A_{ij}=1$, we define $\tilde A_{ij}=1$.  For entries $(i,j)$
such that $A_{ij}=-p/(1-p)$, we take $\tilde A_{ij} =-n_j/(n-n_j)$, where
$n_j$ is the number of $1$'s in column $j$ of $A$.
Then there exist $c_1>0$ and $c_2\in(0,1)$ depending on $p$ such that
\begin{equation}
P(\Vert A-\tilde A\Vert_F^2 \le c_1 N) \ge 1-(2/3)^{N}-Nc_2^{n}.
\label{eq:bigprob}
\end{equation}
\label{thm:AtildeA}
\end{theorem}

\subsection{A bound on $\|\tilde S\|$ in the randomized case}

Suppose that the random graph $G = (V,E)$ containing $k$-disjoint-clique subgraph $G^*$ composed of cliques $C_1, \dots, C_k$
is constructed according to ($\Omega_1$) and ($\Omega_2$) with
probability $p$. Let $C_{k+1},$ $r_{1}, \dots, r_{k+1}$, $\hat r$, and $N$ be defined as in Section~\ref{sec: notation}.
Further, let $\theta = 1-p$ in \eref{eqn: y, z defn 2} and $\gamma = p/(1-p)$ in \eref{sigma 4}.
We begin by stating the main theorem of the section.

\begin{theorem} \label{thm: rc S bound}
     Suppose that $G = (V,E)$ has a $k$-disjoint-clique subgraph $G^*$ composed of the cliques $C_1, \dots, C_k$ and
     let $C_{k+1} := V\setminus (\cup_{i=1}^k C_i)$.
     Let $r_i = |C_i|$ for all $i=1,\dots, k+1$ and suppose that  $r_q \le \hat{r}^{3/2}$ for all $q=1, 2,\dots, k$ where $\hat{r} = \min_{i=1,\dots, k} \{r_i\}$.
     Then there exists some $\beta_1, \beta_2 >  0$ depending only on $p$ such that
     \begin{equation} \label{eqn: rc S bound}
          \|\tilde S \|
               \le \beta_1 \left( \sum_{s=1}^k r_s^2 \right)^{1/2} \left( \sum_{q=1}^k \frac{1}{r_q} \right)^{1/2} + \beta_2 \sqrt{N}
     \end{equation}
     with probability tending exponentially to 1 as $\hat{r}$ approaches $\infty$.
\end{theorem}
This theorem is meant to be used in conjunction with
Theorem~\ref{thm: sufficient condition}.  In particular,
if the right-hand side of \eref{eqn: rc S bound} is less than $\hat r$,
then the planted graph $G^*$ may be recovered.

It is clear from \eref{eqn: suff condn} and the second
term on the right-hand side of \eref{eqn: rc S bound}
that the sufficient conditions for uniqueness and optimality given by Theorem~\ref{thm: sufficient condition}
cannot be satisfied unless $N = O(\hat r^2)$. 
We now give a few examples of values for $r_1,\ldots,r_{k+1}$ that fulfill
\eref{eqn: rc S bound}.
\begin{enumerate}
\item
Consider the case $k=1$, i.e., a single large clique.   In this
case, taking $r_1=\mbox{const}\cdot N^{1/2}$ satisfies \eref{eqn: suff condn}
since the first term on the right is $O(N^{1/4})$.

\item
Suppose $k>1$ and $r_1=\ldots=r_k=\mbox{const}\cdot N^\alpha$.
In this case, the first parenthesized factor on the right in
\eref{eqn: rc S bound} is $O(k^{1/2}N^{\alpha})$ while the
second is $O(k^{1/2}N^{-\alpha/2})$, and therefore the first term
is $O(kN^{\alpha/2})$.  For \eref{eqn: suff condn} to hold,
we need this term to be $O(\hat r)=O(N^\alpha)$, which is valid
as long as $k = O( N^{\alpha/2})$.  We also need $\alpha \ge 1/2$ as noted
above to handle the second term on the right. For example, for $\alpha=1/2$
the algorithm can find as many as $O(N^{1/4})$ cliques of this size.
For $\alpha = 2/3$, the algorithm can find as many as $O(N^{1/3})$
cliques of this size, which is the maximum possible since the
cliques are disjoint and $N$ is the number of nodes.

\item
The cliques may also be of different sizes.  For example, if
there is one large
clique of size $O(N^{2/3})$ and $N^{1/6}$ smaller cliques
of size $O(N^{1/2})$, then $\hat r=O(N^{1/2})$, the first parenthesized
factor in \eref{eqn: rc S bound} is $N^{2/3}$ while the second is
$N^{-1/6}$, so the entire first factor is $O(N^{1/2})=O(\hat r)$.
\end{enumerate}

We note that the results for random noise in the $k$-disjoint-clique problem
are much better than the results for adversary-chosen noise.  In the
case of adversary-chosen noise, the number of allowable noise edges is
bounded above by a constant times the number of edges in the smallest clique.
In the case of random noise, the number of allowable noise edges is bounded above by a scalar
multiple of the number of potential edges in the complement of the planted $k$-disjoint-clique subgraph.
Thus, the number of allowable random noise edges can be as much as the square of the number of edges in the smallest clique
(e.g., if there are $N^{1/4}$
cliques each of size $N^{1/2}$, then the smallest clique has
$O(N)$ edges versus $O(N^2)$ random noise edges).

We do not know whether the bound in \eref{eqn: suff condn} is the
best possible.  For instance, there is no obvious barrier preventing
the algorithm from recovering
as many as $N^{1/2}$ planted cliques of size $N^{1/2}$
in a random graph, but our analysis does not carry through to this case.

The remainder of this section is devoted to the proof of
Theorem~\ref{thm: rc S bound}.
We write $\tilde S$ as
$$
     \tilde S = \tilde S_1 + \tilde S_2 + \tilde S_3 + \tilde S_4 + \tilde S_4^T
$$
where $\tilde S_i \in \R^{N\times N}$, $i=1,\dots, 4$ are $(k+1)$ by $(k+1)$ block matrices such that
\begin{align*}
     [\tilde S_1]_{C_q, C_s} &= \left\{ \begin{array}{ll}
                                             \tilde S_{C_q, C_s}, & \mbox{if } q, s \in \{1,\dots, k\}, \; q \neq s \\
                                             0, &\mbox{otherwise}
                                        \end{array} \right. \\
     [\tilde S_2]_{C_q, C_s} &= \left\{ \begin{array}{ll}
                                             R_{C_q, C_s}, & \mbox{if } q, s \in \{1, \dots, k\} \\
                                             \hat{S}_{C_q, C_{k+1}}, & \mbox{if } s = k+1 \\
                                             \hat{S}_{C_{k+1}, C_s}, & \mbox{if } q = k+1 \\
                                             \tilde{S}_{C_{k+1}, C_{k+1}}, & \mbox{if } q = s = k+1
                                   \end{array}
                                   \right. \\
     [\tilde S_3]_{C_q, C_s} &= \left\{ \begin{array}{ll}
                                             -R_{C_q, C_s}, & \mbox{if } q, s \in \{1, \dots, k\} \\
                                             0, & \mbox{otherwise}
                                   \end{array}
                              \right. \\
     [\tilde S_4]_{C_q, C_s} &= \left\{ \begin{array}{ll}
                                             \tilde S_{C_q, C_{k+1} } - \hat{S}_{ C_q, C_{k+1} }, &\mbox{if } s = k+1, \; q \in \{1,\dots, k\} \\
                                             0, & \mbox{otherwise}
                                        \end{array} \right.
\end{align*}
where $R \in \R^{N\times N}$ is a symmetric random matrix with independently identically distributed entries such that
$$
     R_{ij} = \left\{    \begin{array}{ll}
                              -1, & \mbox{with probability } p \\
                              p/(1-p), & \mbox{with probability } 1-p
                         \end{array}
               \right.
$$
and $\hat S \in \R^{N\times N}$ such that
$$
     \hat{S}_{ij} = \left\{   \begin{array}{ll}
                                   -1, & \mbox{if } (i,j) \in E \\
                                   p/(1-p), & \mbox{otherwise.}
                              \end{array}
     \right.
$$
Notice that, by Theorem~\ref{Furedi-Komlos}, there exists some $\kappa_1, \kappa_2, \kappa_3 > 0$ such that
\begin{equation} \label{eqn: S_2, S_3 bound}
     P\Big(\|\tilde{S}_2\| + \|\tilde{S}_3\| \ge \kappa_1 \sqrt{N} \Big) \le \kappa_2 \exp(-\kappa_3 N^{1/6}).
\end{equation}
Moreover, by Theorem~\ref{thm: 2.4}, there exists $\kappa_4 >0$ and $\kappa_5, \kappa_6, \in (0,1)$ such that
\begin{equation}
     P\Big(\|\tilde{S}_4\| \ge \kappa_4 \sqrt{N} \Big) \le \kappa_5^N + N \kappa_6^N.
\end{equation}
Hence, there exists some scalar $\beta_2$ depending only on $p$  such that
$$
     \|\tilde S\| \le \|\tilde S_1 \| + \beta_2 \sqrt{N}
$$
with probability tending exponentially to 1 as $\hat{r} \ra \infty$.
It remains to prove that
$$
     \|\tilde S_1 \| = O \left( \left( \sum_{s=1}^k r_s^2 \right)^{1/2} \left( \sum_{q=1}^k \frac{1}{r_q} \right)^{1/2} \right)
$$
with probability tending exponentially to 1 as $\hat{r}$ approaches $\infty$.

To do so, consider two vertex sets $C_q$ and $C_s$ such that $q \neq s$.
Without loss of generality we may assume that $r_q \le r_s$.
Define $H = H_{q,s}$, $D= D_{q,s}$, $F = D_{s,q}$, $\b = \b^{q,s}$, $c =c_{q,s}$, $\y$, $\z$, $A$, and $P$ as in Section~\ref{sec: notation}.
The following theorem provides an upper bound on the spectral norm of $\tilde{S}_{C_q, C_s}$ for $q \neq s$, that holds with probability
tending exponentially to 1 as $\hat{r}$ approaches $\infty$.

\begin{theorem} \label{thm: M bound rc}
     Suppose that $r_q$ and $r_s$ satisfy
     \begin{equation} \label{clique size assumption}
          r_q \le r_s \le r_q^{3/2}.
     \end{equation}
     Then there exists $\tilde B_1 > 0$ depending only on $p$ such that
     \begin{equation}
          \|[\tilde S_1]_{C_q, C_s}\| = \|\tilde{S}_{C_q, C_s}\| \le \tilde B_1 \frac{r_s}{\sqrt{r_q}}
     \end{equation}
     with probability tending exponentially to $1$ as $\hat{r}$ approaches $\infty$.
\end{theorem}

The remainder of this section is devoted to the proof of Theorem~\ref{thm: M bound rc};
Theorem~\ref{thm: rc S bound} will then be established as an immediate consequence.

Recall that $\tilde{S}_{C_q, C_s} = H \circ (\y \e^T + \e \z^T) - c (\e\e^T - H).$
We begin by showing that $A+P$ is nonsingular and, hence, $\y$ and $\z$ are uniquely determined.
Let $\delta := (1-p)/(2p).$
Recall that $n_i^s = r_s - D_{ii}$ corresponds to $r_s$ independent Bernoulli trials each succeeding with probability equal to $p$ and, hence,
     \begin{equation} \label{D chernoff}
          P(n_i^s \ge (1 + \delta)p r_s) = P(r_s - D_{ii} \ge (1 + \delta)p r_s) \le \left( \frac{e^\delta}{(1+\delta)^{(1+\delta)}} \right)^{p r_s}
     \end{equation}
     for each $i \in C_q$ by Theorem~\ref{CH ineq}.
     Rearranging, we have that $D_{ii} \ge (\theta - \delta p) r_s$ with probability at least
     $$
          1 - \left( \frac{e^\delta}{(1+\delta)^{(1+\delta)}} \right)^{p r_s}
     $$
     for each $i \in C_q$.
     Similarly,
     \begin{equation} \label{F chernoff}
        P (n_i^q \le (1+\delta) p r_q) =  P(F_{ii} \ge (\theta - \delta p) r_q)    \ge 1 -  \left( \frac{e^\delta}{(1+\delta)^{(1+\delta)}} \right)^{p r_q}
     \end{equation}
     for all $i\in C_s$.
     Therefore, by the union bound, $r_s - D_{ii} \le (1 + \delta)p r_s$ for all $i \in C_q$ and $r_q - F_{ii} \le (1 + \delta) p r_q$ for all $i \in C_s$,
     and, hence, $D$, $F$ are nonsingular with probability at least
     \begin{align} 
          1 - &r_s \left( \frac{e^\delta}{(1+\delta)^{(1+\delta)} }\right)^{p r_q} - r_q \left( \frac{e^\delta}{(1+\delta)^{(1+\delta)} }\right)^{p r_s} \notag \\
          &\ge 1 - (r_q + r_s) \left( \frac{e^\delta}{(1+\delta)^{(1+\delta)}} \right)^{p \hat{r}} . \label{eqn: CB prob totals}
     \end{align}
     Moreover, applying \eref{eqn: min eig bound perturb} shows that $D+\theta \e\e^T$ and $F+\theta \e\e^T$ are nonsingular
     and
     \begin{align}
          \|(D+\theta \e\e^T)^{-1}\| &\le \|D^{-1}\| \le \frac{1}{(\theta - \delta p) r_s},   \label{eqn: norm D bound rc}\\
          \|(F+\theta \e\e^T)^{-1}\| &\le \|F^{-1}\| \le \frac{1}{(\theta - \delta p) r_q},  \label{eqn: norm F bound rc}
     \end{align}
     with probability at least \eref{eqn: CB prob totals}.
     It follows immediately that $A$ is nonsingular and
     \begin{align} 
          \| A^{-1} \| &= \max\{ \|(D+\theta \e\e^T)^{-1}\|, \|(F+\theta \e\e^T)^{-1}\| \} \notag \\
               & \le \frac{1}{(\theta - \delta p) \min\{r_q, r_s\}} = \frac{1}{(\theta - \delta p) r_q} \label{eqn: A bound rc}
     \end{align}
     with probability at least \eref{eqn: CB prob totals}.

     Recall that, in the case that $A$ is nonsingular, it suffices to prove that $\|A^{-1}\| \|P\| < 1$ to show that $A+P$ is nonsingular.
     Moreover, recall that $\theta = 1-p$ is chosen to ensure that the entries of $H-\theta \e\e^T$ have expected value equal to 0.
     We can extend $H - \theta \e \e^T$ to an $r_s \times r_s$ random matrix $\tilde P$ with entries i.i.d. with expected value equal to 0
     by adding $r_s - r_q$ rows with entries i.i.d. such that each additional entry takes value equal to $-\theta$ with probability $p$ and
     value equal to $p$ with probability $1-p$.
     Therefore, by Theorem \ref{geman tail bound}
     \begin{equation} \label{eqn: P bound rc}
        \|P\| = \|H - \theta \e\e^T\| \le \|\tilde P\| \le  \gamma_1\sqrt{ r_s}
    \end{equation}
     for some $\gamma_1 > 0$ depending only on $p$
     with probability at least $1 - \bar c_1 \exp(-\bar c_2 r_s^{\bar c_3})$ where $\bar{c}_i >0$ depend
     only on $p$.
     Combining \eref{eqn: A bound rc}, \eref{eqn: P bound rc}, \eref{clique size assumption} and applying the union bound shows that
     $$
          \|A^{-1}\| \|P\| = \frac{\gamma_1 \sqrt{r_s}}{(\theta - \delta p) r_q} < 1
     $$
     with probability at least
     $$
          1 - (r_q + r_s) \left( \frac{e^\delta}{(1+\delta)^{(1+\delta)}} \right)^{p \hat{r}} - \bar c_1 \exp(-\bar c_2 r_s^{\bar c_3})
     $$
     for sufficiently large $r_q$.
     Therefore, $A+P$ is nonsingular and $\y$ and $\z$ are uniquely determined with probability tending exponentially to $1$ as $\hat r \ra \infty$.

For the remainder of the section we assume that $A+P$ is nonsingular.
We define $Q,$ $Q_1,$ $Q_2$, $\bar \y$ and $\bar \z$ as in Section~\ref{sec: notation}.
To find an upper bound on $\|\tilde{S}_{C_q, C_s}\|$,
we decompose $\tilde{S}_{C_q, C_s}$ as
$$
     \tilde{S}_{C_q, C_s} = M_1 + M_2
$$
where
$
     M_1 := H \circ (\bar \y \e^T + \e \bar \z^T) - c(\e\e^T - H)
$ and
$
     M_2 := H \circ(Q_1 \b \e^T + \e \b^T Q_2^T)
$

We first obtain an upper bound on the norm of $M_1$.
We define $\d \in \R^{C_q}$ to be the vector such that the entry $d_i$ is the difference between the number of edges added between the node $i$ and $C_s$ and the expected number of such edges for each $i \in C_q$. That is,
$$
     \d = \n_{C_q}^s - E[\n_{C_q}^s]  = \n_{C_q}^s - p r_s \e.
$$
Similarly,  we let $\f:= \n_{C_s}^q - p r_q \e$.
Note that, by our choice of $\d$ and $\f$, we have $r_s I  - D = p r_s I + \Diag(\d)$ and $r_q I - F = p r_q I  + \Diag(\f)$.
Notice that for $\theta = 1- p$ we have $D = \theta r_s I  - \Diag(\d)$.
Expanding (\ref{eqn: ybar formula}) we have
\begin{align*}
     \bar{\y} &= \left(D^{-1} - \frac{\theta D^{-1} \e\e^T D^{-1}}{1 + \theta \e^T D^{-1} \e} \right) \b_1 \\
          &= \frac{D^{-1}}{1 + \theta \e^T D^{-1} \e} (\b_1 + \theta \b_1 \e^T D^{-1} \e - \theta \e\e^T D^{-1} \b_1) \\
          &= \frac{D^{-1}}{1+\theta \e^T D^{-1} \e} (\b_1 + \theta (\b_1 \e^T - \e \b_1^T) D^{-1} \e)
\end{align*}
since $\e^T D^{-1} \b_1 = \b_1^T D^{-1} \e$.
Substituting $\b_1 = c(r_s \e - \bar \d)$, where $\bar \d := \diag(D)$, we have
$$
     \b_1 \e^T - \e \b_1^T = c(\e \bar \d^T - \bar \d \e^T)
$$
and, hence,
\begin{align*}
     \bar\y &= \frac{cD^{-1}}{1 + \theta \e^T D^{-1} \e} (r_s \e - \bar \d + \theta( \e \bar \d^T - \bar \d \e^T) D^{-1}\e) \\
          &= \frac{cD^{-1}}{1 + \theta \e^T D^{-1} \e}(r_s \e - \bar \d + \theta \e \e^T \e - \theta \bar \d \e^T D^{-1} \e) \\
          &=\frac{cD^{-1}}{1 + \theta \e^T D^{-1} \e}(r_s \e + \theta r_q \e) - c \e \\
          &= \frac{c(r_s + \theta r_q) D^{-1}}{1 + \theta \e^T D^{-1} \e} \left(\frac{1}{\theta r_s}(D+ \Diag(\d))\right) \e - c \e \\
          &= \left( \frac{c(r_s + \theta r_q)}{(1+\theta \e^T D^{-1} \e)\theta r_s} - c \right)\e + \frac{c(r_s + \theta r_q)}{(1 + \theta \e^T D^{-1} \e)\theta r_s} D^{-1} \d
\end{align*}
since
$$
     I = \frac{1}{\theta r_s}(D + \Diag(\d)).
$$
Let $\bar \y_1 := \omega_1 \e$, $\bar \y_2 := \upsilon_1 \e$ where
\begin{align*}
     \omega_1 &= \frac{c(\theta r_q +  r_s)}{\theta (r_s + r_q)} - c \\
     \upsilon_1 &= \frac{c(\theta r_q +  r_s)}{(1 + \theta \e^T D^{-1} \e) \theta r_s} - c - \omega_1
\end{align*}
and let
$$
     \bar{\y}_3 := \frac{c(\theta r_q +  r_s)}{(1 + \theta \e^T D^{-1} \e) \theta r_s} D^{-1} \d.
$$
Hence, $\bar \y = \bar\y_1 + \bar \y_2 + \bar \y_3$.
Similarly, $\bar \z = \bar \z_1 + \bar \z_2 + \bar \z_3$ where
$\bar \z_1 := \omega_2 \e$, $\bar \z_2 := \upsilon_2 \e$ where
\begin{align*}
     \omega_2 &= \frac{c(r_q + \theta r_s)}{\theta (r_s + r_q)} - c \\
     \upsilon_2 &= \frac{c(r_q + \theta r_s)}{(1 + \theta \e^T F^{-1} \e) \theta r_q} - c - \omega_2
\end{align*}
and
$$
     \bar{\z}_3 := \frac{c(r_q + \theta r_s)}{(1 + \theta \e^T F^{-1} \e) \theta r_q} F^{-1} \f.
$$
Therefore, we can further decompose $M_1$ as $M_1 = \tilde M_1 + \tilde M_2 + \tilde M_3$
where
\begin{align*}
     \tilde M_1 &:= H \circ( \bar \y_1 \e^T + \e \bar \z_1^T) - c (\e\e^T - H), \\
     \tilde M_2 &:= H \circ( \bar \y_2 \e^T + \e \bar \z_2^T), \;\;\;\;
     \tilde M_3 := H \circ( \bar \y_3 \e^T + \e \bar \z_3^T).
\end{align*}
Notice that the matrix $\tilde M_1$ has entries corresponding
to edges equal to $-c$ and remaining entries equal to $cp/(1-p)$
since
$$
     \omega_1 + \omega_2 = \frac{c (1 + \theta)(r_q + r_s)}{\theta(r_q + r_s)} - 2c = \frac{cp}{\theta}.
$$
Therefore, each entry of the matrix $\tilde M_1$ has
expected value equal to $0$. Moreover, each entry of the random block matrix $\hat M$ of the form
$$
     \hat M = \left[ \begin{array}{cc} \tilde M_1 \\ \tilde R \end{array} \right]
$$
has expected value equal to $0$ if $\tilde R$ has identically independently distributed entries such that
$$
     \tilde R_{i,j} = \left\{ \begin{array}{ll} -c, &\mbox{with probability } p \\ cp/(1-p), & \mbox{with probability } 1-p. \end{array}  \right.
$$
Therefore,
there exists $c_1, c_2, c_3, c_4 > 0$ such that
\begin{equation} \label{eqn: y1 bound}
    \|\tilde M_1 \| \le \| \hat M\| \le c_4 \sqrt{r_s}
\end{equation}
with probability at least $1 - c_1 \exp(-c_2 r_s^{c_3})$ by Theorem~\ref{geman tail bound}.

Next, to obtain upper bounds on $\|\tilde M_2 \|$ and $\|\tilde M_3\|$
we will use the following lemma.

\begin{lemma} \label{Bernstein analysis}
     There exists $B > 0$ depending only on $p$ such that
     \begin{equation} \label{eqn: Bernstein Cq}
          \sum_{i \in C_q} \frac{|\theta r_s - D_{ii}|^\alpha}{D_{ii}} \le B \frac{r_q}{r_s^{1-\alpha/2}}
     \end{equation}
     and
     \begin{equation} \label{eqn: Bernstein Cs}
          \sum_{i \in C_s} \frac{|\theta r_q - F_{ii}|^\alpha}{F_{ii}} \le B \frac{r_s}{r_q^{1-\alpha/2}}
     \end{equation}
     for $\alpha = 1,2$
     with probability at least
     \begin{equation} \label{eqn: Bernstein prob}
         1 - (r_q + r_s) v_p^{\hat{r}} - 2 (2/3)^{\hat{r}}
     \end{equation}
     where $v_p = (e^\delta/(1+\delta)^{(1+\delta)})^p$ and $\delta = \min\{p, \sqrt{p} - p\}.$
\end{lemma}

\begin{proof}
     We first prove \eref{eqn: Bernstein Cq}.
     For each $j \in C_q$, let $n_j := n_j^s$.
     The random numbers $\{n_j: j \in C_q\}$ are independent, and each is the result of $r_s$
     Bernoulli trials, each with probability of success equal to $p$.
     We define $\Psi$ to be the event that at least one $n_j$ is very far from its expected value.
     That is, $\Psi$ is the event that there exists $j \in C_q$
     such that $n_j > t r_s$, where $t = \min \{\sqrt{p}, 2p\}$.
     Moreover,  we define $\tilde \Psi$ to be its complement, and let $\tilde \psi(n_j)$ be the indicator function such that
     $$
          \tilde \psi(n_j) = \left\{ \begin{array}{ll} 1, & \mbox{if }n_j \le t r_s \\ 0, &\mbox{otherwise.} \end{array} \right.
     $$
     Let $B$ be a positive scalar depending on $p$ to be determined later.
     Then
     \begin{align}
          P \left( \sum_{i \in C_q} \right. &\left. \frac{|\theta r_s - D_{ii}|^\alpha}{D_{ii}} \ge B \frac{r_q}{r_s^{1-\alpha/2}}  \right) 
          = P \left( \sum_{i\in C_q} \frac{|n_i - p r_s|^\alpha}{r_s - n_i} \ge B \frac{r_q}{r_s^{1-\alpha/2}} \right) \notag \\
                &\le P \left( \sum_{i\in C_q} \frac{|n_i - p r_s|^\alpha}{r_s - n_i} \ge B \frac{r_q}{r_s^{1-\alpha/2}} \wedge \tilde \Psi \right) + P (\Psi). \label{eqn: M3 bound 2}
     \end{align}
     We will analyze the two terms separately.
     For the first term we use a technique of Bernstein (see \cite{Hoeffding}).
     Let $\phi$ be the indicator function of the nonnegative reals.
     Then,
     \begin{align*}
           P  \left( \sum_{j\in C_q}\right. & \left. \frac{|n_j - p r_s|^\alpha}{r_s - n_j} \ge  B \frac{r_q}{r_s^{1-\alpha/2}} \wedge \tilde \Psi \right) \notag \\
           &= P\left(\sum_{j \in C_q} \frac{|n_j - p r_s|^\alpha}{r_s - n_j} -B \frac{r_q}{r_s^{1-\alpha/2}} \ge 0 \wedge \tilde \psi(n_j) = 1 \;\; \forall \; j\in C_q \right) \\
            &= P\left(\sum_{j \in C_q} \frac{r_s^{1-\alpha/2}|n_j - p r_s|^\alpha}{r_s - n_j} - B r_q \ge 0 \wedge \tilde \psi(n_j) = 1 \;\; \forall \; j\in C_q \right) \\
               &= E\left(\phi\left(\sum_{j \in C_q} \frac{r_s^{1-\alpha/2}|n_j - p r_s|^\alpha}{r_s - n_j} - B r_q \right) \cdot \prod_{j \in C_q} \tilde \psi(n_j) \right).
     \end{align*}
     Let $h$ be a positive scalar depending $p$ to be determined later.
     Notice that for any $h > 0$ and all $x\in \R$, $\phi(x) \le \exp(hx)$.
     Thus, by the independence of the $n_j$'s,
     \begin{align*}
          P \left( \sum_{i\in C_q}\right. &\left. \frac{|n_i - p r_s|^\alpha}{r_s - n_i} \ge  B \frac{r_q}{r_s^{1-\alpha/2}} \wedge \tilde \Psi \right) \\
          &\le E\left( \exp \left( h \sum_{j \in C_q} \left( \frac{|n_j - p r_s|^\alpha}{r_s^{\alpha/2-1}(r_s - n_j)} - B \right) \right)
               \cdot \prod_{j \in C_q} \tilde \psi(n_j) \right) \\
          &= \prod_{j \in C_q} E \left( \exp \left( h \left( \frac{|n_j - p r_s|^\alpha}{r_s^{\alpha/2-1}(r_s - n_j)} - B \right) \right) \tilde\psi(n_j)\right) \\
          &= f_1 \cdots f_{r_q}
     \end{align*}
     where
     $$
          f_j = E \left( \exp \left( h \left( \frac{|n_j - p r_s|^\alpha}{r_s^{\alpha/2-1}(r_s - n_j)} - B \right) \right) \tilde\psi(n_j)\right).
     $$
     We now analyze each $f_j$ individually. Fix $j \in C_q$. Then
     \begin{align*}
          f_j &= \sum_{i=0}^{\lfloor t r_s \rfloor} \exp\left( h \left( \frac{|n_j - p r_s|^\alpha}{r_s^{\alpha/2-1}(r_s - n_j)} - B\right)\right) P(n_j = i) \\
              &\le  \sum_{i=0}^{\lfloor t r_s \rfloor} \exp\left( h \left( \frac{|i - pr_s|^\alpha}{(1-\sqrt{p}) {r_s}^{\alpha/2}} - B\right)\right) P(n_j = i)
     \end{align*}
     since $i \le t r_s$ and, hence, $i \le \sqrt{p} r_s$.
     We now reorganize this summation by considering $i$ such that $|i - pr_s| < \sqrt{r_s}$, then $i$ such that
     $\sqrt{r_s} \le |i - p r_s| < 2 \sqrt{r_s}$ and so on.
     Notice that, since $i \le tr_s \le 2p r_s$, we need only to consider intervals until $|i - p r_s|$ reaches $p r_s$. Hence,
     \begin{small}
     \begin{align*}
          f_j &\le \sum_{k=0}^{\lfloor p \sqrt{r_s} \rfloor} \sum_{i: |i-p r_s| \in [k\sqrt{r_s}, (k+1) \sqrt{r_s})}
                     \exp\left( h \left( \frac{|i - pr_s|^\alpha}{(1-\sqrt{p})r_s^{\alpha/2}} - B\right)\right) P(n_j = i) \\
               &\le \sum_{k=0}^{\lfloor p \sqrt{r_s} \rfloor} \sum_{i: |i-p r_s| \in [k\sqrt{r_s}, (k+1) \sqrt{r_s})}
                     \exp\left( h \left( \frac{(k+1)^\alpha}{1-\sqrt{p}} - B\right)\right) P(n_j = i) \\
               &\le 2 \sum_{k=0}^{\lfloor p \sqrt{r_s}  \rfloor}
                    \exp\left( h \left( \frac{(k+1)^\alpha}{1-\sqrt{p}} - B\right)\right) \exp(-k^2/p)
     \end{align*}
     \end{small}
     by \eref{Chernoff bound sqrt root}.
     Overestimating the finite sum with an infinite sum, we have
     \begin{align*}
          f_j \le 2 \exp(-hB) \cdot \sum_{k=0}^\infty \exp\left( \frac{h(k+1)^\alpha}{1-\sqrt{p}} - k^2/p \right).
     \end{align*}
     Choosing $h$ such that $h \le (1-\sqrt{p})/(8p)$ ensures that
     $$
          \frac{h(k+1)^\alpha}{1-\sqrt{p}} - k^2/p \le -k^2/(2p)
     $$
     for all $r_q, r_s$ and $k \ge 1$.
     Hence, splitting off the $k=0$ term, we have
     \begin{align} 
          f_j \le 2 \exp & \left( \frac{h}{1-\sqrt{p}} - hB \right) 
               + 2 \exp(-hB) \cdot \sum_{k=1}^\infty \exp(-k^2/(2p)) \label{eqn: f_j bound}.
     \end{align}
     Since $\sum_{k=1}^\infty \exp(-k^2/(2p))$ is dominated by a geometric series, the summation in \eref{eqn: f_j bound} is a finite
     number depending on $p$.
     Therefore, once $h$ is chosen, it is possible to choose $B$, depending only on $p$ and $h$, sufficiently large so that each of the two terms in
     \eref{eqn: f_j bound} is at most $1/3$.
     Therefore, we can choose $h$ and $B$ so that $f_j \le 2/3$ for all $j \in C_q$.
     It follows immediately that
     \begin{equation} \label{eqn: 1st term bound}
          P\left( \sum_{i \in C_q} \frac{|\theta r_s - D_{ii}|^\alpha}{D_{ii}} \ge B \frac{r_q}{r_s^{1-\alpha/2}} \wedge \tilde \Psi \right)  \le (2/3)^{r_q} \le (2/3)^{\hat{r}}.
     \end{equation}
     To obtain a bound on the second term in \eref{eqn: M3 bound 2}, notice that the probability that $n_j > tr_s$ is at most
     $v_p^{r_s} \le v_p^{\hat{r}}$ where $v_p = (e^\delta/(1+\delta)^{(1+\delta)})^p$ by Theorem~\ref{CH ineq}, where $\delta = t/p - 1 = \min\{p, \sqrt{p} -p\}$.
     Thus, applying the union bound shows that
     $$
          P\left( \sum_{i \in C_q} \frac{|\theta r_s - D_{ii}|^\alpha}{D_{ii}} \ge B \frac{r_q}{r_s^{1-\alpha/2}}  \right) \le (2/3)^{\hat{r}} + r_q v_p^{\hat{r}}.
     $$
     By an identical argument
     $$
          P \left( \sum_{i \in C_s} \frac{|\theta r_q - F_{ii}|^\alpha}{F_{ii}} \ge B \frac{r_s}{r_q^{1-\alpha/2}} \right) \le (2/3)^{\hat{r}} + r_s v_p^{\hat{r}}.
     $$
     Applying the union bound one last time shows that \eref{eqn: Bernstein Cq} and \eref{eqn: Bernstein Cs} hold simultaneously with probability at least
     $
          1 - (r_q + r_s) v_p^{\hat{r}} - 2 (2/3)^{\hat{r}}
     $
     as required.
     \qed
     \bigskip
\end{proof}

As an immediate corollary of Lemma~\ref{Bernstein analysis}, we have the following bound on $|\upsilon_1|$ and $|\upsilon_2|$.

\begin{cor} \label{lemma: nu bounds}
     There exists $B_1 > 0$ depending only on $p$ such that
     $$
          |\upsilon_1| + |\upsilon_2| \le { B_1} \frac{r_q^{3/2} + r_s^{3/2}}{(r_q + r_s)(r_q r_s)^{1/2}}
     $$
     with probability at least $1 - (r_q + r_s) v_p^{\hat{r}} - 2 (2/3)^{\hat{r}}.$
\end{cor}

\begin{proof}
     We begin with $\upsilon_1$. Notice that
     \begin{align*}
          \upsilon_1 &= \frac{c(\theta r_q + r_s)}{\theta} \left( \frac{1}{(1 + \theta \e^T D^{-1} \e) r_s} - \frac{1}{r_q + r_s} \right) \\
               &= \frac{c(\theta r_q + r_s)(r_q - \theta r_s \e^T D^{-1} \e)}{\theta r_s (r_q + r_s)(1+ \theta \e^T D^{-1} \e)}.
     \end{align*}
     Moreover,
     \begin{align*}
          |\theta r_s \e^T D^{-1} \e - r_q|
               = \left|\sum_{i\in C_q} \frac{\theta r_s}{D_{ii}} - \sum_{i\in C_q} 1 \right|
               \le \sum_{i\in C_q} \left| \frac{\theta r_s}{D_{ii}} - 1 \right|
               = \sum_{i\in C_q} \frac{|\theta r_s - D_{ii}|}{D_{ii}}
     \end{align*}
     and, since $D_{ii} \le r_s$ for all $i \in C_q$, we have
     \begin{equation} \label{denominator bound}
          r_s (1 + \theta \e^T D^{-1} \e) \ge r_s \left( 1 + \frac{\theta r_q}{r_s} \right) = \theta r_q + r_s.
     \end{equation}
     Therefore, setting $\alpha = 1$ in \eref{eqn: Bernstein Cq} shows that
     \begin{align}
          |\upsilon_1| &\le \frac{c(\theta r_q + r_s) B r_q}{\theta r_s^{3/2} (r_q + r_s)(1 + \theta \e^T D^{-1} \e)} \label{eqn: nu2 bound1} \\
              &\le \frac{cB (\theta r_q + r_s) r_q}{\theta \sqrt{r_s}(r_q + r_s)(\theta r_q + r_s)} \notag  \\
               &\le B_1 \frac{ r_q}{\sqrt{r_s}(r_q + r_s)}   \notag
     \end{align}
     where $B_1 := B/\theta$ and
     where (\ref{eqn: nu2 bound1}) holds with probability at least
     $
          1 - (2/3)^{\hat{r}} - r_q v_p^{\hat{r}}.
     $
     By an identical calculation
     \begin{equation} \label{eqn: nu4 bound}
          |\upsilon_2 |\le \frac{B_1 {r_s}}{ \sqrt{r_q}(r_q + r_s)}
     \end{equation}
     with probability at least $1 - (2/3)^{\hat{r}} - r_s v_p^{\hat{r}}$.
     Applying the union bound completes the proof.
     \qed
     \bigskip
\end{proof}

Observe that, as an immediate consequence of Corollary~\ref{lemma: nu bounds} and the facts that
$H \circ \e\e^T = H$ and $\|H\|_F \le \sqrt{r_q r_s}$, we have
\begin{align} 
     \|\tilde M_2\| &= \|H \circ(\bar \y_2 \e^T + \e \bar \z_2^T)\| 
          \le (|\upsilon_1| + |\upsilon_2|) \|H\|_F \notag   \\ 
          &\le B_1 \frac{{r_q}^{3/2} + {r_s}^{3/2}}{r_q + r_s} 
          \le 2 B_1 \sqrt{r_s} \label{eqn: y2 bound}
\end{align}
with probability at least $1 - (r_q + r_s) v_p^{\hat{r}} - 2 (2/3)^{\hat{r}}$.

The following corollary of Lemma~\ref{Bernstein analysis} provides an an upper bound on $\|\tilde M_3\|$.

\begin{cor}
     \label{lemma: M3 bound 1}
     There exists $B_2$ depending only on $p$ such that
     \begin{equation} \label{M3 1}
         \|\tilde M_3\| \le  \|H \circ (\bar \y_3 \e^T + \e \bar\z_3^T) \| \le B_2 ( \sqrt{r_q} + \sqrt{r_s})
     \end{equation}
     with probability at least $1 - (r_q + r_s) v_p^{\hat{r}} - 2 (2/3)^{\hat{r}}$.
\end{cor}

\begin{proof}
     To obtain an upper bound on $\|\tilde M_3\|$, we first obtain upper bounds on $\|H\circ(\bar \y_3 \e^T)\|$
     and $\|H \circ(\e \bar \z_3^T)\|$.
     We begin with $\|H\circ(\bar \y_3 \e^T)\|$.
     Since
     $$
          \sum_{i \in C_q} D_{ii} (D^{-1} \d)^2_i = \sum_{i \in C_q} \frac{|\theta r_s - D_{ii}|^2}{D_{ii}}
     $$
     applying \eref{eqn: Bernstein Cq} with $\alpha = 2$ and \eref{chain 1} with $W = H$, $\u = \bar \y_3$, and $\v = \e$ shows that
\begin{align}
     \|H\circ(\bar \y_3 \e^T)\| &\le \left( \sum_{i\in C_q} \bar\y_3(i)^2 \|H(i,:)\|^2 \right)^{1/2} \notag \\
          &= \frac{c(\theta r_q + r_s)}{(1 + \theta \e^T D^{-1} \e) \theta r_s} \cdot \left( \sum_{i \in C_q} D_{ii} [D^{-1} \d]^2_i \right)^{1/2} \notag \\
          &\le \frac{c(\theta r_q + r_s)}{(1 + \theta \e^T D^{-1} \e) \theta r_s} (B r_q)^{1/2} \label{eqn: H3 1} \\
          &\le B_2 \sqrt{r_q}  \label{eqn: H3 2}
\end{align}
where $B_2:= \sqrt{B}/\theta$, \eref{eqn: H3 2} follows from \eref{denominator bound}
and \eref{eqn: H3 1} holds with probability at least $1-(2/3)^{\hat{r}} - r_q v_p^{\hat r}$.
Similarly,
\begin{equation}    \label{eqn: z3 bound}
     \|H \circ(\e \bar \z_3^T)\| \le {B_2}\sqrt{r_s}
\end{equation}
with probability at least $1 - ((2/3)^{\hat{r}} + r_s v_p^{\hat{r}})$.
Applying the union bound shows that
\begin{equation} \label{eqn: M3 bound}
     \|\tilde M_3\| \le \|H\circ(\bar \y_3 \e^T)\| + \|H \circ(\e \bar \z_3^T)\| \le B_2 (\sqrt{r_q} + \sqrt{r_s})
\end{equation}
with probability at least $1 - (r_q + r_s) v_p^{\hat{r}} - 2 (2/3)^{\hat{r}}$ as required.
\qed
\bigskip
\end{proof}

We complete the proof of Theorem~\ref{thm: M bound rc} by showing that $M_2 = H \circ(Q_1 \b\e^T + \e \b^T Q_2^T)$
has norm at most a constant multiple of $r_s/\sqrt{r_q}$ with high probability.
The following lemma provides an upper bound on $\|Q_1 \b\|$ and $\|Q_2 \b \|$.

\begin{lemma}
     \label{Qb norm}
     There exists $B_3,$ $B_4$ and $\bar c_i >0$, $i=1,2,3$, depending only on $p$
     such that
     \begin{equation}
          \label{eqn: Q1b rc}
          \|Q_1 \b\| \le B_3 \frac{r_s^{1/2}}{r_q^{1/2}}
     \end{equation}
     \begin{equation}
          \label{eqn: Q2b rc}
          \|Q_2 \b\| \le B_4 \frac{r_s^{1/2}(r_q + r_s^{1/2})}{r_q^{3/2}}
     \end{equation}
     with probability at least
     \begin{equation}\label{Qb prob}
          1 - \bar c_1 \exp(-\bar c_2 \hat{r}^{\bar c_3}) - (r_q + r_s) \left( \frac{e^\delta}{(1+\delta)^{(1+\delta)}} \right)^{p \hat{r}}
     \end{equation}
     where $\delta = (1-p)/(2p)$.
\end{lemma}

\begin{proof}
     We first derive a bound on each of $\|Q_1\|$, $\|Q_2\|$ and $\|\b\|$ and consequently a bound
     on each of $\|Q_1 \b\|$ and $\|Q_2 \b\|$ by applying the inequalities
     $$
          \|Q_1 \b\| \le \|Q_1\| \|\b\| \;\;\; \mbox{and} \;\;\; \|Q_2 \b\| \le \|Q_2\| \|\b\|.
     $$
     Recall that
     $$
          \|Q_1\| \le \|(D+\theta \e\e^T)^{-1}\| \sum_{\ell = 1}^\infty \|P_1 P_2\|^\ell
               + \|P_1\|\|(F + \theta \e\e^T)^{-1}\| \sum_{\ell = 0}^\infty  \|P_1 P_2\|^\ell
     $$
     and
     $$
          \|Q_2\| \le \|(F+\theta \e\e^T)^{-1}\| \sum_{\ell = 1}^\infty \|P_1 P_2\|^\ell
               + \|P_2\|\|(D + \theta \e\e^T)^{-1}\| \sum_{\ell = 0}^\infty  \|P_1 P_2\|^\ell
     $$
     where
     \begin{align*}
          P_1 &= (D+\theta \e\e^T)^{-1} (H - \theta \e\e^T), \;\;\;\;
          P_2 = (F +\theta \e\e^T)^{-1} (H^T - \theta \e\e^T).
     \end{align*}
     Applying the upper bounds on $\|(D+\theta \e\e^T)^{-1}\|$, $\|(F + \theta \e\e^T)^{-1}\|$, and $\|H - \theta \e\e^T\|$
     given by \eref{eqn: norm D bound rc}, \eref{eqn: norm F bound rc}, and \eref{eqn: P bound rc}
     shows that
     \begin{equation}   \label{eqn: P1 P2 bound rc}
          \|P_1 P_2\| \le \frac{\|H-\theta \e\e^T\|^2}{(\min_{i\in C_q} D_{ii})(\min_{i \in C_s}{F_{ii}})}
               \le \frac{\gamma_1^2}{(\theta - \delta p)^2 r_q}
     \end{equation}
     with probability at least \eref{Qb prob}.
     Therefore, there exists $\gamma_2> 0$ depending only on $p$  such that
     \begin{align}
          \|Q_1\| &\le \frac{1}{(\theta - \delta p)r_s} \sum_{\ell = 1}^\infty \left( \frac{\gamma_1^2}{(\theta - \delta p)^2 r_q}\right)^{\ell}
               + \frac{\gamma_1}{(1-\delta p)^2 r_q \sqrt{r_s}} \sum_{\ell=0}^{\infty} \left( \frac{\gamma_1^2}{(\theta - \delta p)^2 r_q}\right)^{\ell}  \notag \\
               &\le \frac{\gamma_2}{r_q \sqrt{r_s}}       \label{eqn: Q1 bound rc}
     \end{align}
     with probability at least \eref{Qb prob}
     since
     $$
          \sum_{\ell = 0}^{\infty}\left( \frac{\gamma_1^2}{(\theta - \delta p)^2 r_q}\right)^{\ell} \le O(1)
     $$
     and
     $$
          \sum_{\ell = 1}^{\infty} \left( \frac{\gamma_1^2}{(\theta - \delta p)^2 r_q}\right)^{\ell} \le O(r_q^{-1}).
     $$
     with probability at least \eref{Qb prob} in the case that
     $
          r_q >  (\gamma_1/(\theta - \delta p))^2.
     $
     Similarly, there exists $\gamma_3>0$ depending only on $p$ such that
     \begin{equation} \label{eqn:  Q2 bound rc}
          \|Q_2\| \le \frac{\gamma_3}{r_q}(r_q^{-1} + r_s^{-1/2}) = \frac{\gamma_3(r_q + r_s^{1/2})}{r_q^2 r_s^{1/2}}
     \end{equation}
     with probability at least \eref{Qb prob}.
     Finally, recall that
     $$
          [\b]_i = c \cdot \left\{ \begin{array}{ll} n_i^s, & \mbox{if } i \in C_q \\ n_i^q, & \mbox{if } i \in C_s. \end{array} \right.
     $$
     Therefore, by \eref{D chernoff} and \eref{F chernoff}
     \begin{align*}
          \|\b\| &= c \left( \sum_{i\in C_q} (n_i^s)^2 + \sum_{i\in C_s}(n_i^q)^2 \right)^{1/2} \le (1 + \delta) p c ( r_q r_s)^{1/2} ( r_q + r_s)^{1/2}
     \end{align*}
     with probability at least $1 - (r_q + r_s) \left( {e^\delta}/{(1+\delta)^{(1+\delta)}} \right)^{p \hat{r}}$.
     Thus, applying the union bound shows that there exists $B_3, B_4$ depending only on $p$ such that
     \begin{align*}
          \|Q_1 \b \| & \le  \frac{\gamma_2(1+\delta)p c {(r_q r_s)}^{1/2}(r_q + r_s)^{1/2}}{ r_q r_s^{1/2}} \le B_3 \frac{r_s^{1/2}}{r_q^{1/2}}\\
          \|Q_2 \b \| & \le \frac{\gamma_3 (1+\delta)p c {(r_q r_s)}^{1/2}(r_q + r_s)^{1/2}(r_q + r_s^{1/2})}{r_q^{2} r_s^{1/2}}
                              \le B_4 \frac{r_s^{1/2}(r_q + r_s^{1/2})}{r_q^{3/2}}
     \end{align*}
     with probability at least \eref{Qb prob} since $r_q \le r_s$ by Assumption~\eqref{clique size assumption}.
     \qed
     \bigskip
\end{proof}

Finally, to obtain an upper bound on $\|M_2\|$ we decompose $M_2$ as
$$
     M_2 = (H - \theta \e\e^T)\circ (Q_1 \b \e^T) + \theta Q_1 \b \e^T + (H - \theta \e\e^T) \circ (\e (Q_2 \b)^T) + \theta \e (Q_2 \b)^T.
$$
As an immediate corollary of Lemma~\ref{Qb norm} we have
\begin{equation} \label{W2 norm 1}
     \|(Q_1 \b) \e^T\| \le \|Q_1 \b\| \|\e\| = \sqrt{r_s} \|Q_1 \b\| \le B_3 \frac{r_s}{r_q^{1/2}}
\end{equation}
and
\begin{equation} \label{W2 norm 2}
     \|\e  (Q_2 \b)^T\| \le \|\e\| \|Q_2 \b\| = \sqrt{r_q}\|Q_2 \b\| \le B_4 \frac{r_s^{1/2}(r_q + r_s^{1/2})}{r_q}
\end{equation}
with probability at least (\ref{Qb prob}).
Moreover, applying \eref{chain 0} with $W = H - \theta \e\e^T$, $\u = Q_1\b$, and $\v = \e$
we have
\begin{equation} \label{H Q1 b bound}
     \| (H - \theta \e\e^T)\circ (Q_1 \b \e^T) \| \le \|H - \theta \e\e^T\| \|Q_1 \b\|_\infty \le \|H - \theta \e\e^T\| \|Q_1 \b\|.
\end{equation}
Thus, combining \eref{H Q1 b bound}, \eref{eqn: P bound rc}, and \eref{eqn: Q1b rc} we have
\begin{equation} \label{H odot Q b e 1}
     \| (H - \theta \e\e^T) \circ (Q_1 \b \e^T) \| \le B_3 \gamma_1 \frac{r_s}{r_q^{1/2}}
\end{equation}
with probability at least (\ref{Qb prob}).
Similarly,
\begin{equation} \label{H odot Q b e 2}
     \| (H - \theta \e\e^T) \circ (\e (Q_2 \b)^T) \| \le B_4 \gamma_1 \frac{r_s(r_q + r_s^{1/2})}{r_q^{3/2}}
\end{equation}
with probability at least (\ref{Qb prob}).
Therefore, there exists $\hat c$ depending only on $p$ such that
\begin{equation} \label{M2 bound rc}
     \| M_2 \| \le \hat c \frac{r_s}{r_q^{1/2}}
\end{equation}
with probability at least \eref{Qb prob}
since $r_s \le r_q^{3/2}$ by \eqref{clique size assumption}, and, hence $(r_q + r_s^{1/2})/r_q \le O(1)$.

Combining \eref{eqn: y1 bound},  \eref{eqn: y2 bound}, \eref{eqn: M3 bound},
and (\ref{M2 bound rc}),
there exists some $\tilde B_1$ depending only on $p$  such that
\begin{equation} \label{eqn: M bound}
     \|\tilde{S}_{C_q, C_s}\| \le \tilde B_1 \frac{r_s}{{r_q}^{1/2}} 
\end{equation}
for all $q,s \in \{1,\dots, k\}$, $q \neq s$
with probability tending exponentially to $1$ as $\hat{r}$ approaches $\infty$.
The proof of Theorem~\ref{thm: M bound rc} is now complete.
Theorem~\ref{thm: rc S bound} is an immediate consequence:
from \eqref{eqn: M bound} 
it follows that
\begin{align*}
     \sum_{q = 1}^k \sum_{s = 1}^k \| \tilde{S}_{C_q, C_s}\|^2 \le  2 \tilde B_1^2 \sum_{q=1}^k \sum_{s=1}^k \frac{r_s^2}{r_q} 
          =  2 \tilde B_1^2 \left( \sum_{s=1}^k r_s^2 \right) \left( \sum_{q=1}^k \frac{1}{r_q} \right)
\end{align*}
and, hence, there exists some $\beta_1$ depending only on $p$ such that
$$
     \| \tilde{S}_1 \| \le \beta_1 \left( \sum_{s=1}^k r_s^2 \right)^{1/2} \left( \sum_{q=1}^k \frac{1}{r_q} \right)^{1/2}
$$
as required.

\section{Experimental Results}
\label{sec: simulations}
In this section, we evaluate the performance of our relaxation of the $k$-disjoint-clique problem for a variety
of program inputs via simulation.

We consider $N$-node random graphs $G$ constructed according to $(\Omega_1)$ and $(\Omega_2)$ containing planted cliques
$C_1, \dots, C_k$  as follows. For fixed minimum clique size $\hat r$, we choose $k = \lfloor{N/\hat r}\rfloor$ and
choose the clique sizes such that $r_i \in [\hat r, 2\hat r]$.
Every node is assigned to exactly one of the $k$ planted cliques.
We set each entry of the adjacency matrix $A_G$ in the diagonal blocks 
indexed by $C_1, \dots, C_k$ equal to $1$ and independently choose each remaining upper triangular entry of $A_G$ to be equal to $1$
with probability $p$. 
The lower triangular entries of $A_G$ are chosen by symmetry.
We solve the semidefinite program
$$
	\max \{ \tr(X \e\e^T): X\succeq 0,\; X\e \le \e, \; \tr(X) =k, \; X_{ij} = 0 \; \mbox{if } [A_G]_{ij} = 0, \; i\neq j\}
$$
using \verb|SDPNAL| \cite{SDPNAL} in Matlab. 
We declare the planted $k$-disjoint-clique subgraph to be recovered if the optimal solution $X^*$ returned by
\verb|SDPNAL| satisfies $\|X^*-X_0\|_F/\|X_0\|_F < 10^{-4}$,
where $X_0$ is the predicted optimal solution given by \eqref{eqn: proposed solution}.
This experiment was repeated $10$ times each for different values of $p$ and $\hat r$ for $N=200$ and $N=500$.
We plot the empirical probability of  recovery of the planted $k$-disjoint clique subgraph for each set of trials in
Figure~\ref{fig: results}.

\begin{figure}
	\caption{Simulation results for $N$-node graph with $k$ planted cliques of roughly identical size.
	We plot the average number of recoveries of the planted cliques per set of $10$ trials for different minimum cluster sizes $\hat r$ and different
	probabilities of adding noise edges. Brighter colours indicate a higher rate of recovery.}
	\label{fig: results}
	\centering
	\subfloat[{$N=200$}]{\includegraphics[width=0.5\textwidth]{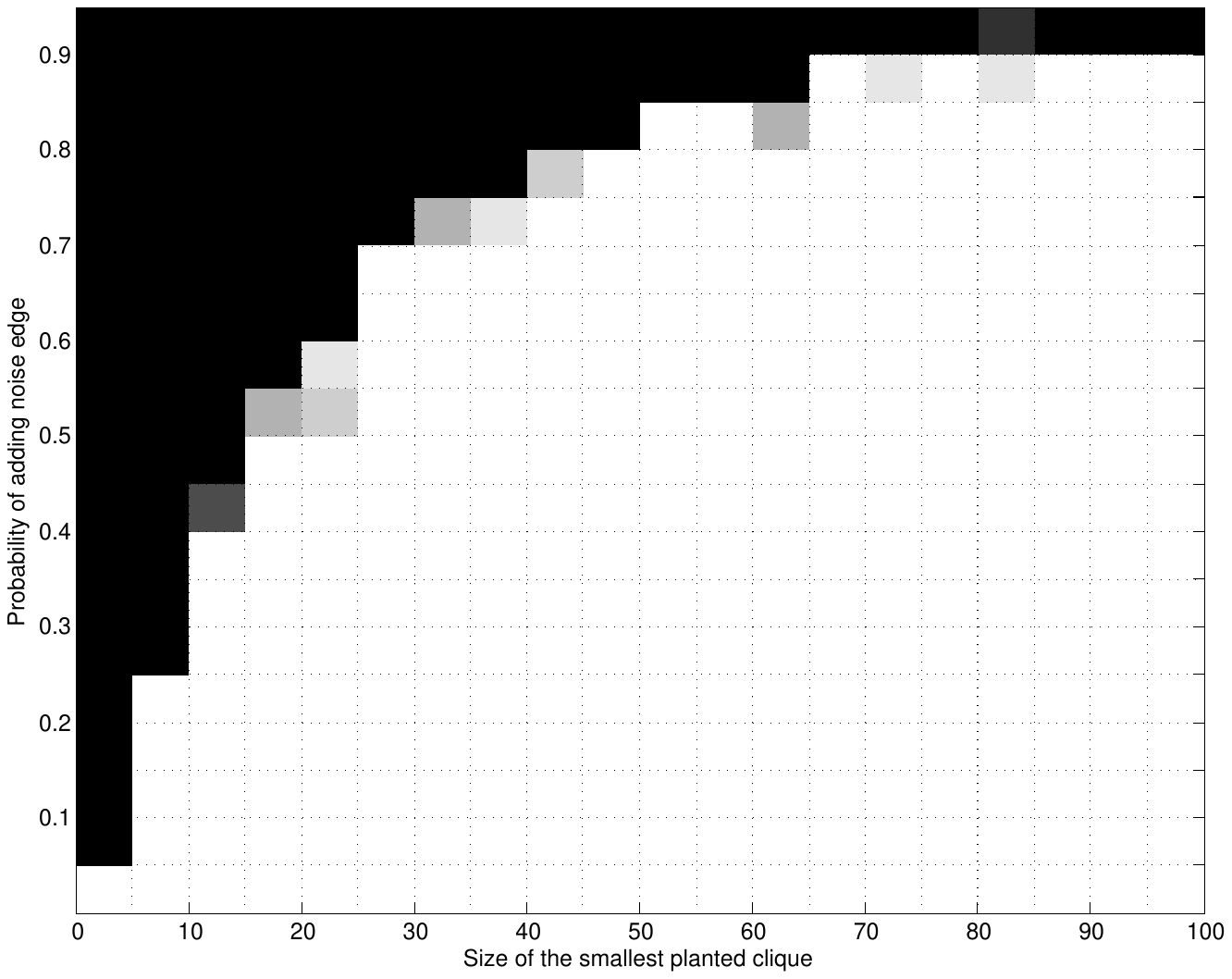} }  
	\subfloat[$N=500$]{\includegraphics[width=0.5\textwidth]{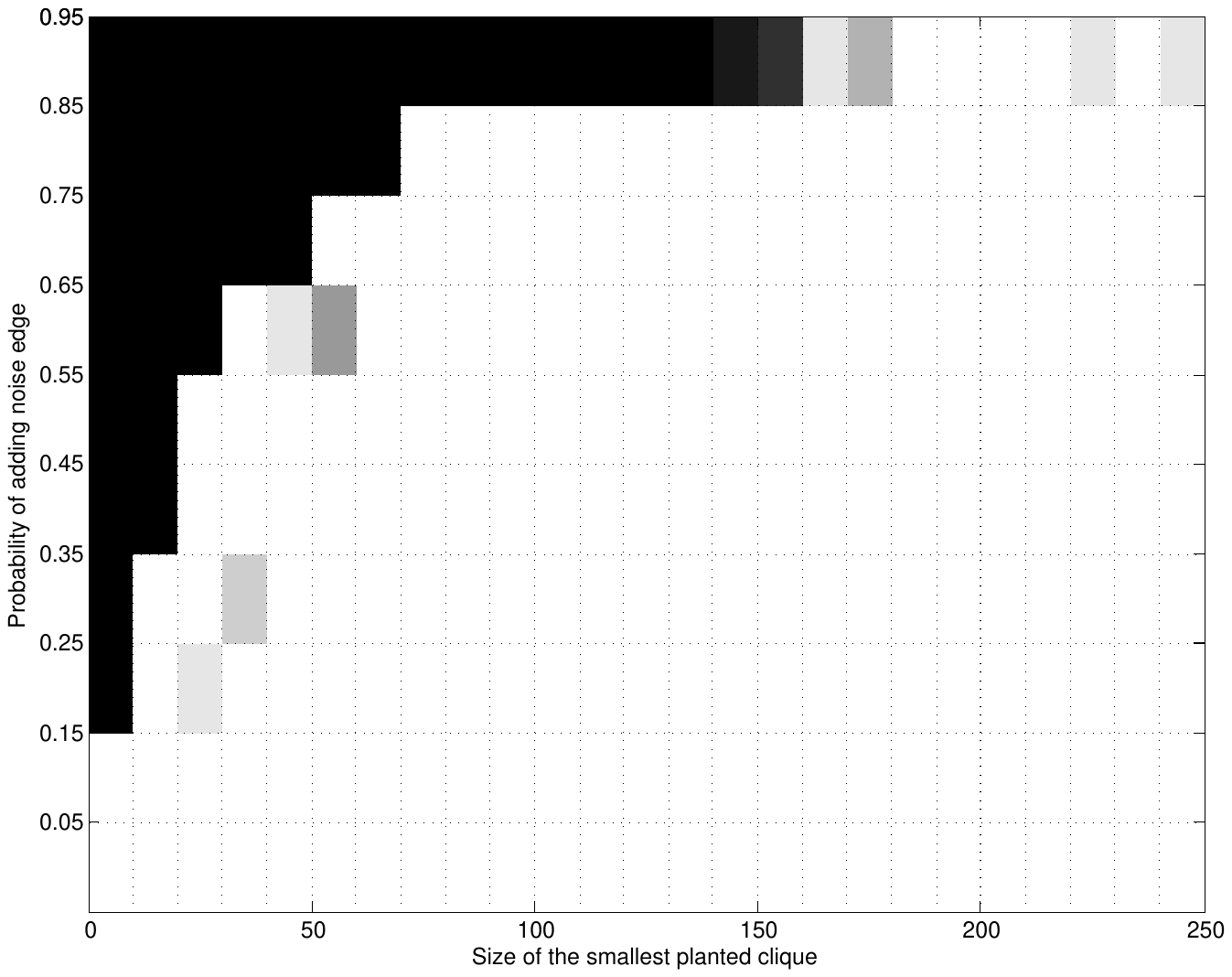}} 
\end{figure}

The performance of our heuristic closely matches that predicted by Theorem~\ref{thm: rc S bound}, although the constants in the bounds on the sizes of planted cliques may be conservative.
For example, the bounds ensuring recovery of $X_0$ given by Theorem~\ref{thm: rc S bound}  cannot be satisfied if
the minimum clique size $\hat r$ is less than $3(p/1-p) \sqrt{N}$, however perfect recovery occurs for smaller
values of $\hat r$ in almost all trials.

\section{Conclusions}
\label{sec: conclusions}
We have considered an NP-hard combinatorial version of the clustering
problem called the $k$-disjoint-clique problem in which input data is
an undirected graph.  We have shown that a convex relaxation of the
problem can exactly solve the problem for input instances constructed
in a certain way.  The construction of the instance is that $k$
disjoint cliques are first placed in the input graph, and then many
`noise' vertices and edges are placed that obscure the $k$ disjoint
cliques.  We have shown that the algorithm exactly recovers the clique
for noise edges placed by an adversary provided the conditions stated
in Theorem~\ref{thm: adv case} on the number of noise edges are
satisfied.  In the case of random noise, many more noise edges and
nodes can be tolerated compared to the adversary case; in particular,
if the quantity on the right-hand side of Theorem~\ref{thm: rc S bound} 
is at most $\hat r$, then the algorithm recovers the
planted cliques with probability exponentially close to 1.

This work raises several open questions.  First, as already noted in the
text, our bounds may not be the best possible.  Particularly in the
random case, there is nothing in the way of matching lower bounds.

Another open question is whether the techniques developed herein can
be applied to other formulations of clustering.  For example, if
clustering is posed as an optimization problem with a distance function,
then can an approach like the one described here find the optimal
solution for input instances constructed in a certain way?


\bibliographystyle{plain}
\bibliography{paperbib}


\end{document}